\documentclass{article}
\usepackage{amsthm,amsmath,amssymb}
\RequirePackage[numbers]{natbib}
\RequirePackage[colorlinks,citecolor=blue,urlcolor=blue]{hyperref}
\def\bE{\mathbb E}
\def\bE{\mathbb E}
\def\R{\mathbb R}
\def\Q{\theta_{\mathrm{mx}}}

\def\E{\mathbb{E}}

\newtheorem{Theorem}{Theorem}

\newtheorem{lemma}[Theorem]{Lemma}
\newtheorem{Corollary}[Theorem]{Corollary}
\newtheorem{proposition}[Theorem]{Proposition}
\newtheorem{Assumption}{Assumption}
\newtheorem{remark}{Remark}

\def\argmin{\mathop{\rm arg\,min}}

\def\rank{{\rm rank}}

\def\keywords{\vspace{.5em}
{\textit{Keywords}:\,\relax%
}}
\def\endkeywords{\par}

\def\subjclass{\vspace{.5em}
{\textit{AMS 2000 subject classification}:\,\relax%
}}
\def\endsubjclass{\par}

\newcommand{\mX}{X}
\newcommand{\mZ}{Z}
\newcommand{\mB}{B}
\newcommand{\mQ}{Q}

\newcommand{\mtheta}{\theta}
\newcommand{\mTheta}{\Theta}

\newcommand{\dimd}{d}
\newcommand{\dimk}{k}
\newcommand{\dims}{s}

\newcommand{\dimn}{n}
\newcommand{\dimdone}{n}
\newcommand{\dimdtwo}{m}
\newcommand{\dimkone}{k_{\dimdone}}

\newcommand{\dimsone}{s_{\dimdone}}

\begin{document}
%\begin{frontmatter}

\title{Structured Matrix Estimation and Completion}
\author{ Olga Klopp, \textit{ESSEC and CREST}\footnote{kloppolga@math.cnrs.fr}\\% \vskip 0.25 cm%, 
	Yu Lu,\textit{Yale University }\footnote{yu.lu@yale.edu}\\ Alexandre B. Tsybakov, \textit{ENSAE, UMR CNRS 9194 }\footnote{alexandre.tsybakov@ensae.fr}\\ Harrison H. Zhou, \textit{Yale University }\footnote{huibin.zhou@yale.edu}\\
	}
\maketitle
\begin{abstract}
	We study the problem of matrix estimation and matrix completion under a general framework. This framework includes several important models as special cases such as the gaussian mixture model, mixed membership model, bi-clustering model and dictionary learning. We consider the optimal convergence rates in a minimax sense for estimation of the signal matrix under the Frobenius norm and under the spectral norm. As a consequence of our general result we obtain minimax optimal rates of convergence for various special models.
\end{abstract}

\endkeywords

\endsubjclass

%\section{Preliminaries}\label{Preliminaries}

%\section{Preliminaries}\label{Preliminaries}
\section{Introduction}
Over the past decade, there have been considerable interest in statistical inference for high-dimensional matrices. A fundamental   model in this context is the matrix de-noising model, under which one observes a matrix $\theta^{*}+W$ where $\theta^{*}$ is an unknown non-random $n\times m$ matrix of interest, and $W$ is 
a random noise matrix. The aim is to estimate $\theta^{*}$ from such observations. Often in applications a part of elements of $M$ is missing. The problem of  reconstructing the signal matrix $\theta^{*}$ given partial observations of its entries is known as matrix completion problem. There has been an important research in the past years devoted to accurate  matrix completion methods.

In general, the signal $\theta^{*}$ cannot be recovered consistently from noisy and possibly missing observations. If we only know that $\theta^{*}$ is an arbitrary $n\times m$ matrix, the guaranteed error of estimating $\theta^{*}$ from noisy observations can be prohibitively high.  However, if $\theta^{*}$  has an additional structure one can expect to  estimate it with high accuracy from a moderate number of noisy observations. The algorithmic   and analytical tractability of the problem depends on the type of adopted structural model. A popular assumption in the matrix completion literature is that the unknown matrix $\theta^*$ is of low rank or can be well approximated by a low rank matrix. Significant progresses have been made on low rank matrix estimation and completion problems, see e.g., 
\cite{Candes_Tao,candes-plan-noise,Gross, Keshavan, KLT11,wainwright-weighted, foygel-serebro-concentration,klopp_general,Cai_Zhou}.
%\cite{Candes_Tao,candes-plan-noise,Gross, Keshavan, Koltchinskii_Lounici_Tsybakov,klopp_general,Cai_Zhou}. %Low-rank matrix estimation problem is also known as the matrix factorization problem in the literature and they been successfully applied to recommendation systems \cite{koren2009bellkor},  
% in which our goal is to find two low rank matrix $\theta_1 \in \mathbb{R}^{n\times r} $ and $\theta_2 \in \mathbb{R}^{r\times m}$ such that $\theta^* = \theta_1 \theta_2$, where $r \le m$ and $r \le n$. 
% Low-rank matrix estimation and matrix factorization problems have 
However, in several applications, the signal matrix $\theta^*$ can have other than just low rank structure. Some examples are as follows. 
\begin{itemize}
	\item \textit{Biology}. The biological data are sometimes expected to have clustering structures. For example, in the gene microarray data, a large number of gene expression levels are measured under different experimental conditions. It has been observed in the experiments that there is a bi-clustering structure on the genes \cite{cheng2000biclustering}. This means that, besides being of low rank, the gene microarray data can be rearranged to approximately have a block structure. 
	
	% The most widely studied clustering model is the Gaussian mixture model \cite{pearson1894contributions, titterington1985statistical}. Sometimes, both the samples and the features have a clustering structure, where we use could bi-clustering model \cite{hartigan1972direct}.
	
	% The most widely studied clustering model is the Gaussian mixture model \cite{pearson1894contributions, titterington1985statistical}, in which there are $k$ different rows of our signal matrix $\theta^*$ that corresponds to $k$ clusters. Under Gaussian mixture mode, our signal matrix $\theta^*$ can be written as $ZB$, where $Z \in \{0,1\}^{n \times k}$ only has one non-zero entry in the row indicating the cluster label. Sometimes we have cluster structure on both columns and rows, resulting in a bi-clustering model \cite{hartigan1972direct} that the signal matrix $\theta^* = Z_1 B Z_2^T$. Here $Z_1$ and $Z_2$ have the same structure as the $Z$ in Gaussian mixture model. 
	
	\item \textit{Computer Vision}. To capture higher-level features in natural images, it is common to represent data as a sparse linear combination of basis elements \cite{olshausen1997sparse} leading to sparse coding models. Unlike the principle component analysis that looks for low rank decompositions, sparse coding learns useful  representations with  number of basis vectors, which is often greater than the dimension of the data.

	\item \textit{Networks}. In network models, such as social networks or citation networks, the links between objects are usually governed by the underlying community structures. To capture such structures, several block models have been recently proposed with the purpose of explaining the network data \cite{holland1983stochastic, airoldi2008mixed, karrer2011stochastic}. 
\end{itemize}

While there are some successful algorithmic advancements on adapting new structures in these specific applications, not much is known on the fundamental limits of statistical inference for the corresponding models. A few exceptions are the  stochastic block model  \cite{gao2014rate, klopp_graphon} and the bi-clustering model \cite{gao2015optimal}. However, many other structures of signal matrix are not analyzed. 

The aim of this paper is to study a general framework of estimating structured matrices. 
We consider 
a unified model that includes gaussian mixture model,  %\cite{pearson1894contributions}
mixed membership model  \cite{airoldi2008mixed}, bi-clustering model  \cite{hartigan1972direct}, and dictionary learning as special cases. We first study the optimal convergence rates in a minimax sense for estimation of the signal matrix under the Frobenius norm and under the spectral norm from complete observations on the sparsity classes of matrices. Then, we investigate this problem in the partial observations regime (structured matrix completion problem) and study the minimax optimal rates under the same norms. We also establish accurate oracle inequalities for the suggested methods.

\section{Notation}\label{notation}
This section provides a brief summary of the notation used throughout this paper. Let $A,B$ be matrices in $\mathbb{R}^{n\times m}$.
\begin{itemize}
	\item For a matrix $A$, $A_{ij}$ is its $ (i, j)$th entry, $A_{\cdot j}$ is its $j$th column and $A_{i\cdot }$ is its $i$th row.
	\item The scalar product of two matrices $A,B$ of the same dimensions is denoted by $\langle A,B\rangle =\mathrm{tr}(A^{T}B).$
	\item We denote by $\Vert A\Vert_{2}$ the Frobenius norm of $A$ and by $\Vert A\Vert_{\infty}$ the largest absolute value of its entries:  
	$\left\Vert A\right\Vert_{\infty}=\underset{i,j}{\max}\mid A_{ij}\mid$. The spectral norm of $A$ is denoted by $\Vert A\Vert$.
	\item For $x\in \mathbb{R}^k$, we denote by $\Vert x\Vert_0$ its $l_0$-norm (the number of non-zero components of $x$), and by
	$\Vert x\Vert_q$ its $l_q$-norm, $1\le q\le \infty$.
	\item  We denote by $\Vert A\Vert_{0,\infty}$ the largest $l_0$-norm of the rows of  $A\in \mathbb{R}^{n\times k}$:
	$$
	\|A\|_{0,\infty} =\max_{1\leq i\leq p} \|A_{i\cdot}\|_0.
	$$ 
	\item For any $i\in\mathbb{N}$, we write for brevity $[i]=\{1,\dots,i\}$.
	\item   Given a matrix $A=(A_{ij})\in \R^{n\times m}$, and a set of indices $I\subset [n]\times[ m]$, we define the  restriction of $A$ on $I$ as a matrix $A_{I}$ with elements $\left (A_{I}\right )_{ij}=A_{ij}$ if $(i,j)\in I$ and $\left (A_{I}\right )_{ij}=0$ otherwise. 
\item The notation ${\mathbf I}_{k\times k}$ and ${\mathbf 0}_{k\times l}$ (abbreviated to ${\mathbf I}$ and ${\mathbf 0}$ when there is no ambiguity) stands for the $k\times k$ identity matrix and  the $k\times l$ matrix with all entries 0, respectively.
	\item We denote by $\vert S\vert$ the cardinality of a finite set $S$, by $\lfloor x\rfloor$  the integer part of $x\in \R$, and by $\lceil x\rceil$  the smallest integer greater than $x\in \R$.
	\item  We denote by $\mathcal{N}_{\epsilon}(\mathcal{A})$ the $\epsilon-$covering number, under the Frobenius norm, of a set $\mathcal{A}$ of matrices.
%	\item We set
%	$d=n+m, \quad r_n = n \wedge k_n, \quad r_m= m \wedge k_m.
%	$
	
\end{itemize}
%%%%%%%%%%%%%%%%%%%%%%%%

\section{General model and examples}\label{model_sampling}

Assume that we observe a matrix $Y=(Y_{ij})\in \mathbb{R}^{n\times m}$ with entries
\begin{equation} \label{eq:modelmc}
Y_{ij}=E_{ij}\left (\theta_{ij}^*+\xi_{ij}\right ), \quad i=1,\dots,n, \quad i=1,\dots,m,
\end{equation} 
where $\theta_{ij}^*$ are the entries of the unknown matrix of interest
$\theta^{*}=(\theta_{ij}^*) \in \mathbb{R}^{n\times m}$, the values $\xi_{ij}$ are independent random variables representing the noise, %matrix $W=(W_{ij})\in \mathbb{R}^{n\times m}$,
and $E_{ij}$ are i.i.d. Bernoulli variables with parameter $p\in (0,1]$ such that $(E_{ij})$ is independent of $(\xi_{ij})$.

Model \eqref{eq:modelmc} is called the matrix completion model. Under this model,  an entry of matrix $\theta^{*}$ is observed with noise (independently of the other entries) with probability $p$, and it is not observed with probability $1-p$. We can equivalently write \eqref{eq:modelmc} in the form 
\begin{equation}\label{model}
Y/p=\theta^{*}+W,
\end{equation}
where $W$ is a matrix with entries 
$$W_{ij}= \theta_{ij}^*(E_{ij}-p)/p +\xi_{ij}E_{ij}/p.$$
%when $p$ is known. 
The model with complete noisy observations
is a special case of \eqref{eq:modelmc} (and equivalently of \eqref{model}) corresponding to $p=1$. In this case, $W_{ij}=\xi_{ij}$.

We denote by $\mathbb{P}_{\mtheta^*}$ the  probability distribution of $Y$ satisfying \eqref{eq:modelmc} and by $\mathbb{E}_{\mtheta^*}$ the corresponding expectation. When there is no ambiguity, we abbreviate $\mathbb{P}_{\mtheta^*}$ and $\mathbb{E}_{\mtheta^*}$ to $\mathbb{P}$ and $\mathbb{E}$, respectively.

 We assume that $\xi_{ij}$ are independent zero mean sub-Gaussian random variables. 
%with the sub-gaussian norm $\Vert W_{ij}\Vert_{\psi_2}\leq \sigma$.
The sub-Gaussian property means that the following assumption is satisfied. 
\begin{Assumption}\label{assumption_noise}
	There exists $\sigma>0$ such that, for all $(i,j)\in[n]\times[m]$,
	\begin{align*}
	\forall \ \lambda\in \mathbb{R}, \quad \bE\exp\left (\lambda \xi_{ij}\right )\leq \exp(\lambda^{2}\sigma^{2}/2).
	\end{align*}
\end{Assumption}

We assume that the signal matrix $\theta^{*}$ is structured, that is, it can be factorized using sparse factors. Specifically, let $s_n, k_n, s_m, k_m$ be integers such $0\leq s_n\leq k_n$ and $0\leq s_m\leq k_m$. We assume that $$\theta^{*}\in\Theta(s_n,s_m)\subset \mathbb{R}^{n\times m},$$
where
\begin{align*}
\Theta(s_n,s_m)=\{\theta=XBZ^T\,:\,X\in \mathcal{A}_{s_n},B\in \mathbb{R}^{k_n\times k_m}\;\text{and}\;Z\in \mathcal{A}_{s_m} \}.
\end{align*}  
Here, for $s_n=0$ we assume that $n=k_n$ and the set $\mathcal{A}_{s_n}$ is a set containing only one element, which is the $n\times n$ identity matrix, and for $1\leq s_n\leq k_n$,
\begin{align}\label{asn}
\mathcal{A}_{s_n}=\mathcal{A}_{s_n}(n,k_n)=\{A\in \mathcal{D}_{n}^{n\times k_n},\Vert A_{ i\cdot}\Vert_{0}\leq s_n,\;\text{for all}\; i\in[n]\}
\end{align}  
where the set $\mathcal{D}_{n}$ is a subset of $\mathbb{R}$ called an alphabet. The set $\mathcal{A}_{s_m}$ is defined analogously by replacing $n$ by $m$. 
We will also consider the class $\Theta_{*}(s_n,s_m)$ defined analogously to $\Theta(s_n,s_m)$, with the only difference that the inequality in \eqref{asn} is replaced by the equality.

Choosing different values of $s_n, k_n, s_m, k_m$, and different alphabets we obtain several well-known examples of matrix structures. 

\begin{itemize}
	\item Mixture Model: 
	\begin{eqnarray*}\label{eq:MG}
		\nonumber \Theta_{MM}&=&\{ \theta \in \mathbb{R}^{n \times m}:  \theta = XB \textrm{ for some } B \in \mathbb{R}^{k \times m} \\
		&& \textrm{ and } X \in \{0,1\}^{n \times k} \textrm{ with } \|X_{i\cdot}\|_0=1, \forall i \in [n]\}.
	\end{eqnarray*}
	
	\item Sparse Dictionary Learning: 
	\begin{eqnarray*} \label{eq:SDL}
		\nonumber \mTheta_{SDL}&=&\{ \mtheta=\mB\mZ^T \in \mathbb{R}^{\dimd \times \dimn}: \mB \in \mathbb{R}^{\dimd \times \dimk}, \mZ \in \mathbb{R}^{n \times k}\, \textrm{with}\, \Vert Z_{i \cdot}\Vert_0 \leq s, \forall i \in [n] \}.
	\end{eqnarray*}
	
	\item Stochastic Block Model: 
	\begin{eqnarray*} \label{eq:SBM}
		\nonumber \mTheta_{SBM}&=&\{ \mtheta=Z\mB\mZ^T \in \mathbb{R}^{\dimn \times \dimn}:  \mB \in [0,1]^{\dimk \times \dimk}, \mZ \in \{0,1\}^{n \times k}\, \textrm{with}\, \Vert Z_{i\cdot}\Vert_0 =1, \forall i \in [n] \}.
	\end{eqnarray*}
	
	\item Mixed Membership Model: 
	\begin{eqnarray*} \label{eq:MMM}
		\nonumber \mTheta_{MMM} &=&\{ \mtheta=Z\mB\mZ^T \in \mathbb{R}^{\dimn \times \dimn}: \mB \in [0,1]^{\dimk \times \dimk}, \mZ \in [0,1]^{n \times k},\\
		&& \, \textrm{with}\, \|\mZ_{i\cdot}\|_1=1, \|\mZ_{i\cdot}\|_0 \le s, \textrm{ for all }i \in [\dimn]\}.
	\end{eqnarray*}
	
	\item {Bi-clustering Model:}
	\begin{eqnarray*} \label{eq:BCM}
		\nonumber \mTheta_{Bi} &=&\{ \mtheta=X\mB\mZ^T \in \mathbb{R}^{n \times m}: \mB \in [0,1]^{\dimk_n \times \dimk_m}, X \in \{0,1\}^{n \times k_n}, \mZ \in \{0,1\}^{m\times\dimk_m }\\
		&& \, \textrm{with}\, \|X_{i\cdot }\|_0=1, \forall i \in [n], \|\mZ_{i\cdot }\|_0=1, \forall i \in [m] \}.
	\end{eqnarray*}
	
\end{itemize}

Here, the classes   $\mTheta_{SBM}$ and $\mTheta_{MMM}$ are not exactly equal to but rather subclasses of $\Theta_{*}(1,1)$ 
and $\Theta_{*}(s,s)$, respectively.

Statistical properties of inference methods under the general model \eqref{eq:modelmc} are far from being understood. Some results were obtained in  particular settings such as the Mixture Model and Stochastic Block Model.  

Gaussian mixture models provide a useful framework for several machine learning problems such as clustering, density estimation and classification. There is a quite long  history of research on mixtures of Gaussians.  We mention only some of this work including
%One of the most popular method for estimating a mixture distribution is maximum likelihood which is, unfortunately,  NP-Hard.  This has led to a stream of work on alternative
methods for estimating mixtures such as pairwise distances \cite{dasgupta_mixtures_gaussians,dasgupta_schulman}, spectral methods \cite{vempala_wang,hsu_kakade} or
the method of moments \cite{chaudhuri_dasgupta,belkin_sinha}. Most of these papers are concerned with construction of computationally efficient methods but do not address the issue of statistical optimality. In \cite{Azizyan_Singh} authors provide precise information theoretic bounds
on the clustering accuracy and sample complexity of learning a mixture of two
isotropic Gaussians in high dimensions under small mean separation. 

The Stochastic Block Model is a  useful benchmark for the task of recovering community structure in graph data. More generally, any sufficiently large  graph behaves approximately like a stochastic block model for some $k$, which can be large. The problem of estimation of the probability matrix $\theta^{*}$ in the stochastic block model under the Frobenius norm was considered by several authors \cite{chatterjee2014matrix,WolfeOlhede, xu2014edge, chan2014consistent,borgs2015} but convergence rates obtained there are suboptimal. More recently,  minimax  optimal rates of estimation were obtained by  Gao et al. \cite{gao2014rate} in the dense case and by Klopp et al \cite{klopp_graphon} in the sparse case.

Recently, a  related problem to ours was studied by Soni et al. \cite{soni_factor}. These authors consider the case when the matrix to be estimated is the product of two matrices, one of which, called a sparse factor, has a small number of non-zero entries (in contrast to this, we assume row-sparsity). The estimator studied in \cite{soni_factor} is a sieve maximum likelihood estimator penalized by the $l_0-$norm of the sparse factor where the sieve is chosen as a specific countable set.

%%%%%%%%%%%%%%%%%%%%%%%%%%%%%%%%%%%%%%%%%%%
%%%%%%%%%%%%%%%%%%%%%%%%%%%%%%%%%%%%%%%%%%%%%%%%%%%%%%%%%%%%%%%%%%%%%%%%%%%%%%%%%%%%%%%%%%%%%%%%%%%%%%%%
\section{Results for the case of finite alphabets}\label{main_results}

We start by considering the case of finite alphabets $ \mathcal{D}_{n}$ and $\mathcal{D}_{m}$ and complete observations, that is $p=1$. In this section, we establish the minimax optimal rates of estimation of $\theta^*$ under the Frobenius norm and we show that they are attained by 
 the least squares estimator %(potentially not computational feasible): 
\begin{equation} \label{eq:lsestimator}
\hat {\theta}\in \underset{\theta\in\Theta}{\argmin}\Vert Y-\theta\Vert^{2}_{2}
\end{equation}
where $\Theta$ is a suitable class of structured matrices. We first derive an upper bound on the risk of this estimator uniformly over the  classes $\Theta=\Theta(s_n,s_m)$.  
The following  theorem  provides an oracle inequality for the  Frobenius risk of $\hat {\theta}$. Here and in what follows, we adopt the convention that $0 \log \frac{x}{0}=0$ for any $x>0$. We also set for brevity
$$d=n+m, \quad r_n = n \wedge k_n, \quad r_m= m \wedge k_m.
$$
\begin{Theorem} \label{thm:upperfinite} 
Let Assumption \ref{assumption_noise} hold, and let $p=1$. If the sets $\mathcal{D}_{\dimdone}$ and $\mathcal{D}_{\dimdtwo}$ are finite, there exists a constant $C>0$  depending only on the cardinalities of $\mathcal{D}_{\dimdone}$ and $\mathcal{D}_{\dimdtwo}$ such that, for all $\mtheta^{*}\in \mathbb{R}^{n\times m}$ and all $\epsilon>0$, the risk of the estimator~\eqref{eq:lsestimator} satisfies
\begin{equation*}
\E_{\theta^*}\left\{||\hat{\mtheta} - \mtheta^{*}||_2^2 \right\} \leq (1+\epsilon)\inf_{\bar\theta \in \Theta(s_n,s_m)}\|\bar\theta-\theta^*\|_2^2+\dfrac{C \sigma^2}{\epsilon} ( R_X+R_B+R_Z),
\end{equation*}
where $R_X= nr_m \wedge ns_n\log \frac{ek_n}{s_n}$, $R_B=r_n r_m$, $R_Z=mr_n \wedge ms_m \log \frac{ek_m}{s_m}$.  
\end{Theorem}
This theorem is proved in Section \ref{proof_thm_upperfinite}. 

Note that if the set $\mathcal{A}_{s_n}$ and/or $\mathcal{A}_{s_m}$ in the definition of $\Theta(s_n,s_m)$ contains only the identity matrix, the corresponding term $R_X$ and/or $R_Z$ disappears from the upper bound of Theorem \ref{thm:upperfinite}. 

In Theorem \ref{thm:upperfinite}, the true signal $\theta^*$ can be arbitrary. By assuming that  $\theta^* \in \Theta(s_n,s_m)$, we immediately deduce from Theorem \ref{thm:upperfinite} that the following bound holds.

\begin{Corollary} \label{cor:1}
Under the assumptions of Theorem~\ref{thm:upperfinite},
$$
\sup_{\theta \in \Theta(s_n,s_m)}\E_{\theta} \left\{ ||\hat{\mtheta} - \mtheta||_2^2 \right\} \leq
 C \sigma^2( R_X+R_B+R_Z)
$$
for a constant $C>0$  depending only on the cardinalities of $\mathcal{D}_{\dimdone}$ and $\mathcal{D}_{\dimdtwo}$.
\end{Corollary}

The next theorem provides a lower bound showing that the convergence rate of Corollary~\ref{cor:1} is minimax optimal. This lower bound is valid for the general matrix completion model \eqref{eq:modelmc}. In what follows, the notation $\inf_{\hat{\vartheta}}$ stands for the infimum over all estimators $\hat\vartheta$ taking values in $\R^{n\times m}$.
\begin{Theorem} \label{thm:lowerbound}
Let the entries $W_{ij}$ of matrix $W$ in model (\ref{model}) be  independent random variables with Gaussian distribution $\mathcal{N}(0, \sigma^2)$, and let the alphabets $\mathcal{D}_{\dimdone}$ and $\mathcal{D}_{\dimdtwo}$ contain the set $\{0,1\}$.
There exists an absolute constant $C>0$ such that
\begin{equation} \label{eq:lower}
\inf_{\hat{\vartheta}}\sup_{\theta\in\Theta(s_m,s_n)}\mathbb{P}_{\theta} \left\{ ||\hat{\vartheta} - \theta||_2^2 \geq \frac{C \sigma^2}{p}\left(R_X + R_B + R_Z \right) \right\}\geq 0.1,
\end{equation}
and
\begin{equation} \label{eq:expectlower}
\inf_{\hat{\vartheta}}\sup_{\theta\in\Theta(s_m,s_n)} \mathbb{E}_{\theta}  ||\hat{\vartheta} - \theta||_2^2 \geq \frac{C \sigma^2}{p}\left(R_X + R_B + R_Z \right).
\end{equation}
Furthermore, the same inequalities hold with  $\Theta_*(s_m,s_n)$ in place of $\Theta(s_m,s_n)$ if $s_n\in\{0,1\}$ and $s_m\in\{0,1\}$.
\end{Theorem}

The proof of Theorem \ref{thm:lowerbound} is given in Section \ref{sec:prooflower}. 

The three ingredients $R_X, R_B$, and $R_Z$ of the optimal rate are coming from the ignorance of $\mX,\mB$ and $\mZ$ respectively.  %The main idea of this proof is similar to the lower bound proof in \cite{gao2014rate}. 
The proof is based on constructing subsets of $\Theta$ by fixing two of these parameters to get each of the three terms. The choice of $B$ when fixing the pairs $(X, B)$ and $(Z, B)$ is based on a probabilistic method, namely, Lemma~\ref{lm:specialB}. Similar techniques have been used in \cite{klopp_graphon} to prove the lower bounds for sparse graphon estimation, and in \cite{gao2014rate}.  %We also generalize the lower bounds to the matrix completion case. %as it was previously done  in \cite{gao2015optimal}. \\

\begin{remark}
	Theorem \ref{thm:lowerbound} can be extended to more general sub-Gaussian distributions under
	an additional Kullback-Leibler divergence assumption. Assume that there is a constant $c$ such that the distribution
	of $Y$ in model (\ref{eq:modelmc}) satisfies 
	\begin{equation*}
	KL\left( \mathbb{P}_{\theta },\mathbb{P}_{\theta ^{\prime }}\right) \leq 
	\frac{cp}{2\sigma ^{2}}\Vert \theta -\theta ^{\prime }\Vert _{2}^{2}.
	\end{equation*}
	Let the alphabets $\mathcal{D}_{n}$ and $%
	\mathcal{D}_{m}$ contain the set $\{0,1\}$. Then there exists an absolute constant $C>0$ such that 
	\begin{equation*}
	\inf_{\hat{\vartheta}}\sup_{\theta \in \Theta (s_{m},s_{n})}\mathbb{P}%
	_{\theta }\left\{ ||\hat{\vartheta}-\theta ||_{2}^{2}\geq \frac{C\sigma ^{2}%
	}{p}\left( R_{X}+R_{B}+R_{Z}\right) \right\} \geq 0.1, 
	\end{equation*}%
	and 
	\begin{equation*}
	\inf_{\hat{\vartheta}}\sup_{\theta \in \Theta (s_{m},s_{n})}\mathbb{E}%
	_{\theta }||\hat{\vartheta}-\theta ||_{2}^{2}\geq \frac{C\sigma ^{2}}{p}%
	\left( R_{X}+R_{B}+R_{Z}\right) . 
	\end{equation*}
	The proof of this result is similar to that of Theorem \ref{thm:lowerbound} and only needs to replace the equality in  (\ref{eq:KLZ}) by inequality. In addition, the lower bounds hold with  $\Theta_*(s_m,s_n)$ in place of $\Theta(s_m,s_n)$ if $s_n\in\{0,1\}$ and $s_m\in\{0,1\}$.
\end{remark}

\begin{remark}
	We summarize the minimax rates for some examples introduced in Section \ref{model_sampling} in the following table.
	The case of $\Theta _{SBM}$ is due to \cite{gao2014rate}.
\begin{center}
\begin{tabular}{|c|c|}\hline \vphantom{$\displaystyle\int$}
$\Theta_{MM}$& $\min\{n\log(ek)+km, nm\}$\\
\hline \vphantom{$\displaystyle\int$}
$\Theta _{SDL}$& $\min\{ns\log(ek/s)+kd, nd\}$ \\
\hline \vphantom{$\displaystyle\int$}
$\Theta _{SBM}$& $n\log(ek)+k^2$
\\
\hline \vphantom{$\displaystyle\int$}
$\Theta _{Bi}$& $\min\{n\log(ek_n)+m\log(ek_m)+k_nk_m, nk_m+m\log(ek_m), mk_n+n\log(ek_n),nm\}$
\\
\hline
\end{tabular}
\end{center}
%	
%	
%	
%	
%	  $\min\{n\log(ek)+km, nm\}$ for the Mixture Model $\Theta_{MM}$, $\min\{ns\log(ek/s)+kd, nd\}$ for $\Theta _{SDL}$ the Sparse Dictionary Learning, $n\log(ek)+k^2$ for $\Theta _{SBM}$ the Stochastic Block Model, and $\min\{n\log(ek_1)+m\log(ek_2)+k_1k_2, nk_2+m\log(ek_2), {\color{red}mk_1+n\log(ek_1)},nm\}$ for $\Theta _{Bi}$ the Biclustering model.
	
\end{remark}

%% % % % % % % % % % % % % % % % % % % % % % % % % % % % % % % % % % % % % % % % %
\section{Optimal rates in the spectral norm}\label{section_spectral}
In this section we derive the optimal rates of convergence of estimators of $\theta^*$ when the error is measured in the spectral norm. Interestingly, our results imply that these  optimal rates coincide with those obtained for estimation of matrices with no structure. That is, the additional structure that we consider in the present paper does not have any impact on the rate of convergence of the minimax risk when the error is measured in the spectral norm.

The  lower bound under the spectral norm can be obtained as a corollary of the lower bound under the Frobenius norm given by Theorem \ref{thm:lowerbound}. 
\begin{Corollary} \label{cor:speclower}
Under the assumptions of Theorem \ref{thm:lowerbound},
there exists a absolute constant $C'>0$ such that
\begin{equation*} \label{eq:speclower}
\inf_{\hat{\vartheta}}\sup_{\theta\in\Theta(s_m,s_n)}\mathbb{P}_{\theta} \left\{ ||\hat{\vartheta} - \theta||^2 \geq \frac{C' \sigma^2}{p} (n \vee m) \right\}\geq 0.1,
\end{equation*}
and
\begin{equation*} \label{eq:expectspeclower}
\inf_{\hat{\vartheta}}\sup_{\theta\in\Theta(s_m,s_n)} \mathbb{E}_{\theta}  ||\hat{\vartheta} - \theta||^2 \geq \frac{C' \sigma^2}{p}\left(n \vee m \right).
\end{equation*}
\end{Corollary}
The proof of this corollary  is given in Section \ref{sec:proofspeclower}.

To get matching upper bounds we can use the soft thresholding estimator introduced in \cite{Koltchinskii_Lounici_Tsybakov} or the hard thresholding estimator proposed in \cite{Klopp_rank}. These papers deal with the completion problem for low rank matrices in the context  of trace regression model, which is a slightly different setting. 
%They provide bounds on the estimation error measured in spectral norm  which are optimal up to a log factor for the class of matrices $\mathcal{A}(r,a)$ defined as follows: for any $A_0\in \mathcal{A}(r,a)$ the rank of $A_0$ is supposed not to be larger than a given $r$ and all the entries of $A_0$ are supposed to be bounded in absolute value by a constant $a$.
%In \cite{KLT11} the corresponding lower bound  is obtained for the class of matrices $\mathcal{A}(r,a)$ defined as follows: for any $A_0\in \mathcal{A}(r,a)$ the rank of $A_0$ is supposed not to be larger than a given $r$ and all the entries of $A_0$ are supposed to be bounded in absolute value by a constant $a$.

Here, we consider the hard thresholding estimator. Set 
\begin{equation*}
Y'=Y/p.
\end{equation*} 
The singular value decomposition of matrix $Y'$ has the form
\begin{equation}
Y'=\overset{\rank(Y')}{\underset{j=1}{\Sigma}}\sigma_j(Y')u_j(Y')v_j(Y')^{T},
\end{equation}
where $\rank(Y')$ is the rank of $Y'$,
 $\sigma_j(Y')$ are the singular values of $Y'$ indexed in the decreasing order, and 
$u_j(Y')$ (respectively, $v_j(Y')$) are the left (respectively, the right) singular vectors of $Y'$.
The hard thresholding estimator is defined by the formula
\begin{equation}\label{hard_estimator}
\tilde \theta=\underset{j:\sigma_j(Y')\geq \lambda}{\Sigma}\sigma_j(Y')u_j(Y')v_j(Y')^{T}
\end{equation}
where $\lambda>0$ is the regularization parameter.  In this section, we assume that the noise variables $W_{ij}$ are  bounded as stated in the next assumption.
 \begin{Assumption}\label{noise_bounded} For all $i,j$ we have
$\bE(W_{ij})=0$, $\bE(W_{ij}^{2})=\sigma^{2}$ and there exists a positive constant  $b>0$ such that
\begin{equation*} 
 \underset{i,j}{\max}\left \vert W_{ij}\right \vert\leq b.
 \end{equation*}
 \end{Assumption} 
A more general case of sub-Gaussian noise can be treated as well; in this case, we can work on the event $\mathcal{E}_{b}$ where $\Vert W\Vert_{\infty}$ is bounded by a suitable constant $b$ and show that the probability of the complement of $\mathcal{E}_{b}$ is small.  

The following theorem  gives the upper bound on the estimation error of the hard thresholding estimator \eqref{hard_estimator}.
 \begin{Theorem}\label{corollary_spectral}
 Assume that $\Vert \theta^{*}\Vert_{\infty}\leq \Q$ and let Assumption \ref{noise_bounded} hold. Let $\lambda=c(b+\Q)\sqrt{\frac{n\vee m}{p}}$ where $c>0$ is a sufficiently large absolute constant. Assume that $p\geq \log(n+m)/(n\vee m)$. Then,  with ${\mathbb P}_{\theta^*}$probability at least $1-2/(n+m)$, the hard thresholding estimator $\tilde \theta$  satisfies
 \begin{equation*}
  \Vert \tilde \theta-\theta^{*}\Vert^{2}\leq C(b+\Q)^{2}\dfrac{n\vee m}{p}
 \end{equation*}
 where $C>0$ is an absolute constant. % that depends only on $\sigma,b$ and $\Q$.
 \end{Theorem}
 The proof of Theorem \ref{corollary_spectral} is close to the argument in \cite{Klopp_rank}. It is given in Section~\ref{proof_corollary_spectral}.

 % % % % % % % % % % % % % % % % % % % % % % % % % % % % % % % % % % % % % % %
\section{A general oracle inequality under incomplete observations}
The aim of this section is to present a general theorem about the behavior of least squares estimators in the setting with incomplete observations. This theorem will be applied in the next section to obtain an analog of the upper bound of Theorem~\ref{thm:upperfinite} for general alphabets.  To state the theorem, it does not matter whether we consider a vector or matrix setting.
Therefore, 
in this section, we will deal with the vector model. 
Assume that we observe a vector $Y=(Y_1, \dots, Y_N)$ with entries
\begin{equation} \label{eq:modelmc1}
Y_i=E_i\left (\theta_i^*+\xi_i\right ), \quad i=1,\dots,N, 
\end{equation} 
for some unknown  $\theta^*=(\theta^{*}_1, \dots, \theta^{*}_N)$. Our goal is to estimate $\theta^*$. 
Here, $\xi_i$ are independent random noise variables, and $E_i$ are i.i.d. Bernoulli variables with parameter $p\in (0,1]$ such that $(E_1, \dots, E_N)$ is independent of $(\xi_1, \dots, \xi_N)$. 

When $p$ is known we can equivalently write \eqref{eq:modelmc1} in the form 
\begin{equation}\label{model1}
Y'=\theta^{*}+W,
\end{equation}
where now $W$ is a vector with entries 
$$W_i= \theta_i^*(E_i-p)/p +\xi_{i}E_{i}/p,$$
and $Y'= Y/p$.
In this section, we denote by $\mathbb{P}_{\mtheta^*}$ the  probability distribution of $Y'$ satisfying \eqref{model1}. 

Consider the least squares estimator of  $\theta^{*}$:
\begin{equation} \label{eq:ls} 
\hat{\theta} \in \argmin_{\theta \in \Theta} \| Y' - \theta\|_2^2,
\end{equation}
where $\Theta$ is a subset of $\mathbb{R}^N$.  For some element $\theta_0$ of  $\argmin_{\theta \in \Theta} \| \theta - \theta^* \|_2^2$ we set $\Theta_1 = \left\{ \theta \in \Theta: \|\theta - \theta_0\|_2 \le 1 \right\}$.  

Set
$$
\epsilon_0 = \frac12 \Big( \inf\{\epsilon\in(0,1]: \ N\epsilon^2> \log {\mathcal N}_{\epsilon}(\Theta_1)\} + \sup\{\epsilon\in(0,1]: \ N\epsilon^2< \log {\mathcal N}_{\epsilon}(\Theta_1)\}\Big).
$$
Since $\mathcal N_{\epsilon}(\Theta_1)$ is a decreasing left-continuous function of $\epsilon\in(0,1]$, we have
 %there exists an integer $s^*\geq 2$ and $\epsilon_0<1$ such that
\begin{equation} \label{eq:entropycond} 
\frac12 \log {\mathcal N}_{\epsilon_0}(\Theta_1)\leq N \epsilon_0^2 \leq  \log {\mathcal N}_{\epsilon_0}(\Theta_1).
\end{equation}

 %This estimator has been considered in the matrix completion literature \cite{KLT11, gao2015optimal}. 

%Note that in the case of the sub-exponential error $W$,  existing proof techniques based on Dudley's Entropy bound (for example, \cite{sara_book}) will give us the following result
%$$\|\hat{\theta} - \theta\|_2 \lesssim \int_{0}^{\infty} \sqrt{\log \mathcal{N}_\epsilon(\Theta)} d \epsilon +  \int_{0}^{\infty} \log \mathcal{N}(\epsilon, \Theta, \|\cdot\|_2) d \epsilon.$$
%However, as we will see in Theorem \ref{thm:oracleupper}, this bound is not sharp {\color{red} in  "completion setting"} that we consider in this section. % when the parameter space $\Theta$ is bounded. 

\begin{Theorem} \label{thm:oracleupper}
Let $\xi_i$ be independent random variables satisfying $\mathbb{E} e^{\lambda \xi_i} \le e^{\lambda^2 \sigma^2/2}$ for some $\sigma>0$ and all $\lambda\in \mathbb{R}$.
Assume that there exists a constant $\Q$ such that $\| \theta \|_\infty \le \Q$ for all $\theta \in \Theta$.
 Then, for any $\theta^*\in \mathbb{R}^N$, with $\mathbb{P}_{\theta^*}$-probability at least $1-4/{\mathcal N}_{\epsilon_0}(\Theta_1) - \exp(-pN/6)$,  the least squares estimator (\ref{eq:ls}) satisfies the oracle inequality 
$$ \|\hat{\theta} - \theta^* \|_2^2 \le 3\inf_{\theta \in \Theta}\|\theta-\theta^*\|_2^2 +C \frac{\Q^2+\sigma^2}{p} N \epsilon_0^2,
$$ 
where $C>0$ is an absolute constant.
\end{Theorem}
The proof of this theorem is given in Section \ref{proof_thm:oracleupper}.

Note that Theorem \ref{thm:oracleupper} has no assumption on the true signal $\theta^*$. Using Theorem \ref{thm:oracleupper} with $\theta^*\in \Theta$ we immediately deduce that 
$$\inf_{\theta\in\Theta}\mathbb{P}_{\theta} \left(\|\hat{\theta} - \theta\|_2^2 \le C \frac{\Q^2+\sigma^2}{p} N \epsilon_0^2\right)
\ge 1-4/\mathcal N_{\epsilon_0}(\Theta_1) - \exp(-pN/6),
$$ 
where $C>0$ is an absolute constant.

Theorem \ref{thm:oracleupper} shows that the rate of convergence of the least squares estimator is determined by the value  of $\epsilon_0$ satisfying the global entropy condition (\ref{eq:entropycond}). This quantity is the critical covering radius that appeared in the literature in different contexts, see, e.g., \cite{yang1999information}. In particular, this critical radius has been shown to determine the minimax optimal rates in nonparametric estimation problems. However, it may lead to slightly suboptimal rates (with deterioration by a logarithmic factor) for parametric estimation problems.

\section{Structured matrix completion with general alphabets}
For the structured matrix completion over infinite alphabets  we consider the following parameter spaces:
\begin{align*}
\widetilde\Theta(s_n,s_m)&=\{\theta=XBZ^T: X\in \widetilde{\mathcal{A}}_{n},\Vert B\Vert_{\infty}\le B_{max}, Z\in 
\widetilde{\mathcal{A}}_{m},  
\Vert \theta\Vert_{\infty}\leq \Q
\}.
 \end{align*}  
Here, $B_{max}$ and $\Q$ are positive constants, and for $1\le s_n \le k_n$,
 \begin{align*}
        \widetilde{\mathcal{A}}_{n}=\{A\in \mathcal{D}_{n}^{n\times k_n}: \ \Vert A_{ i\cdot}\Vert_{0}\leq s_n,\;\text{for all}\; i\in[n]\;\text{and}\;\Vert A\Vert_{\infty}\leq 1\}.
\end{align*}
If $s_n=0$, we assume that $n=k_n$ and we define $\widetilde{\mathcal{A}}_{s_n}$ as the set containing only one element, which is the $n\times n$ identity matrix.

 The difference from the class $\Theta(s_n,s_m)$ is only in the fact that the elements of matrix $\theta\in \widetilde\Theta(s_n,s_m)$ and those of the corresponding factor matrices $X,B,Z$ are assumed to be uniformly bounded. This assumption is natural in many situations, for example, in the Stochastic Block Model or in recommendation systems, where the entries of the matrix are ratings.  We introduce the bounds of the entries of the factor matrices in order to fix ambiguities associated with the factorization structure. 
 
A key ingredient in applying Theorem \ref{thm:oracleupper} to this particular case  is to find the covering number $\log \mathcal{N}_\epsilon(\Theta_1)$ when $\Theta=\widetilde\Theta(s_n,s_m)$. For any $\Theta\subset \mathbb{R}^{n\times m}$, any $\theta_0\in \Theta$, and any $u>0$, set
$$\Theta_u = \left\{ \theta \in \Theta: \|\theta - \theta_0\|_2 \le u \right\}.$$
 The following result is proved in Section \ref{proof_prop:coveringbound}.
\begin{proposition} \label{prop:coveringbound}
For any $\theta_0\in \widetilde\Theta(s_n,s_m)$, $0<\epsilon<1$, and $u\le 1$ we have
$$\log \mathcal{N}_\epsilon\left ( \widetilde{\Theta}_u(s_n,s_m) \right ) \le R_1(\epsilon) \wedge R_2(\epsilon) \wedge R_3(\epsilon) \wedge R_4(\epsilon),$$
where
\begin{eqnarray*}
R_1(\epsilon) &=& ns_n \log \frac{ek_n}{s_n}+ m s_m \log \frac{ek_m}{s_m} + (ns_n+ms_m) \log \frac{6B_{max} \sqrt{mn} s_m s_n}{\epsilon} +  r_n r_m \log \frac{9u}{\epsilon},\\
R_2(\epsilon) &=& nr_m \log \frac{6u}{\epsilon} + ms_m \log \frac{ek_m}{s_m} + ms_m \log \frac{2B_{max}\sqrt{mn}s_ms_n}{\epsilon}, \\
R_3(\epsilon) &=& mr_n \log \frac{6u}{\epsilon} + ns_n \log \frac{ek_n}{s_n} + ns_n \log \frac{2B_{max}\sqrt{mn}s_ms_n}{\epsilon}, \\
R_4(\epsilon) &=& m n \log \frac{3u}{\epsilon}.
\end{eqnarray*}
\end{proposition}
Theorem \ref{thm:oracleupper}, together with Proposition \ref{prop:coveringbound}, imply  implies the following upper bound on the estimation error in  structured matrix completion.
\begin{Corollary}\label{cor_gneral_alph}
Consider model~\eqref{eq:modelmc}. Let Assumption 1 hold. Then, for any $\theta^*\in \mathbb{R}^{n\times m}$, the least squares estimator (\ref{eq:lsestimator}) with $\Theta=\widetilde\Theta(s_n,s_m)$ satisfies the  inequality
$$  \|\hat{\theta} - \theta^* \|_2^2 \le 3\inf_{\theta \in \widetilde \Theta(s_n,s_m)}\|\theta-\theta^*\|_2^2 + C \frac{\Q^2+\sigma^2}{p} \left(R_1(\epsilon_0) \wedge R_2(\epsilon_0) \wedge R_3(\epsilon_0) \wedge R_4(\epsilon_0) \right) $$
with $\mathbb{P}_{\theta^*}$-probability at least $$1 - \exp\left(-c(R_1(\epsilon_0) \wedge R_2(\epsilon_0) \wedge R_3(\epsilon_0) \wedge R_4(\epsilon_0))\right) - \exp(-pmn/18),$$ where, $C,c>0$ are absolute constants.
\end{Corollary}
Note that Proposition \ref{prop:coveringbound} and \eqref{eq:entropycond} imply that 
 $\epsilon_0 \geq c'\sqrt{(m+n)/mn}$ for some numerical constant $c'$. Then, we have  that for the general scheme of matrix completion and general alphabets the upper bound given by Corollary \ref{cor_gneral_alph} departs from the lower bound of Theorem \ref{thm:lowerbound} by a logarithmic factor:
\begin{Corollary}\label{cor_gneral_alph_1}
Let the assumptions of Corollary \ref{cor_gneral_alph} be satisfied.
Then the least squares estimator (\ref{eq:lsestimator}) with $\Theta=\widetilde\Theta(s_n,s_m)$ satisfies the  inequality
\begin{align*}
&\inf_{\theta \in \widetilde \Theta(s_n,s_m)}\mathbb P_{\theta}\left(\Vert {\hat \theta}-\theta\Vert ^{2}_{2}\le C\frac{\Q^2+\sigma^2}{p} \left [\log(n+ m)+\log(s_ns_m)\right ]\left(R_X + R_B + R_Z \right)\right) 
\\
&\ge 1 - \exp\big(- c\left [\log(n+ m)+\log(s_ns_m)\right ]\left(R_X + R_B + R_Z \right) \big) - \exp(-pmn/18)
\end{align*}
where $C,c>0$ are absolute constants.
\end{Corollary}
%\appendix

% % % % % % % % % % % % % % % % % % % % % % % % % % % % % % % % % % % % % % % % % % % % % % % % % % % % % % % % % % % % % % % % % % % % % %
\section{Adaptation to unknown sparsity}\label{matrix_bic}
The estimators considered above require the knowledge of 
the degrees of sparsity $s_n$ and $s_m$ of $\theta^{*}$.
In this section, we suggest a method that does not require such a knowledge and thus it is adaptive to the unknown degree of sparsity. Our approach will be to estimate $\theta^{*}$ using a sparsity penalized least squares estimator. 
%To do so, for any couple $(s_n,s_m)\in [k_n]\times [k_m]$ we consider the following sets of matrices
%\begin{align*}
%\widetilde\Theta(s_n,s_m)=\{\theta\in \Theta(s_n,s_m) \quad\text{and}\quad \Vert \theta\Vert_{\infty}\leq \Q\}.
% \end{align*}  
% with
% \begin{align*}
%        \widetilde{\mathcal{A}}_{s_p}=\{A\in \mathcal{D}_{p}^{p\times k_p},\Vert A_{ i\cdot}\Vert_{0}= s_p,\;\text{for all}\; i\in[p]\;\text{and}\;\Vert A\Vert_{\infty}\leq 1\}.
%\end{align*}
Let \begin{align}\label{def_class}
\mathcal{X}=\cup_{s_n=1}^{k_n}\cup_{s_m=1}^{k_m}\widetilde
\Theta(s_n,s_m)
\end{align} 
and set 
%\begin{align*}
%R_1(s_n,s_m)&=r_n r_m \log (9\sqrt{n\wedge m})+ns_n\log\left (B_{max}k_ns_m\sqrt{n\wedge m}/(\sigma\vee \Q)\right )\nonumber\\&\hskip 0.5 cm +ms_m\log\left (B_{max}k_ms_n\sqrt{n\wedge m}/(\sigma\vee \Q)\right ),\nonumber\\
%R_2(s_n,s_m)&=nr_m \log (6\sqrt{n\wedge m}) + ms_m\log\left (B_{max}k_ms_n\sqrt{n\wedge m}/(\sigma\vee \Q)\right ),\nonumber\\
%R_3(s_n,s_m)&=mr_n \log (6\sqrt{n\wedge m}) + ns_n\log\left (B_{max}k_ns_m\sqrt{n\wedge m}/(\sigma\vee \Q)\right ),\nonumber\\
%R_4(s_n,s_m) &= m n \log (3\sqrt{n\wedge m})
%\end{align*} 
%and
\begin{align}\label{def_R_s}
R(s_n,s_m)&=\left [ nr_m\log(6\sqrt{n\wedge m})\right ]\wedge \left [ns_n\log\left (k_ns_m (n\wedge m)\right )\right ]\nonumber\\
&\hskip 0.5 cm+\left [ mr_n\log(6\sqrt{n\wedge m})\right ]\wedge \left [ms_m\log\left (k_ns_m (n\wedge m)\right )\right ]\nonumber\\
&\hskip 1 cm+r_nr_m \log(9\sqrt{n\wedge m}).
\end{align}
For any $\theta=XBZ^T\in \mathcal{X}$ 
let
\begin{align}\label{def_R}
R(\theta)=R(\Vert X\Vert_{0,\infty},\Vert Z\Vert_{0,\infty}).
\end{align}
%\begin{align*} %\label{E_2_G}
%R_1(\theta)&=r_n r_m \log (nm)+n\Vert X\Vert_{0,\infty}\log\left (6eB_{max}k_n\Vert Z\Vert_{0,\infty}nm\right )\nonumber\\&\hskip 0.5 cm +m\Vert Z\Vert_{0,\infty}\log\left (6eB_{max}k_m\Vert X\Vert_{0,\infty}nm\right ),\nonumber\\
%R_2(\theta)&=nr_m \log (nm) + m\Vert Z\Vert_{0,\infty}\log\left (6eB_{max}k_m\Vert X\Vert_{0,\infty}nm\right ),\nonumber\\
%R_3(\theta)&=mr_n \log (nm) + n\Vert X\Vert_{0,\infty}\log\left (6eB_{max}k_n\Vert Z\Vert_{0,\infty}nm\right )\;\text{and}\nonumber\\
%R_4(\theta) &= m n \log (mn).\nonumber
%\end{align*} 
%Finally let
%\begin{align}\label{def_R}
%R(\theta)=R_1(\theta)\wedge R_2(\theta)\wedge R_3(\theta)\wedge R_4(\theta).
%\end{align}
In the following,  $\Omega$ denotes the random set of observed indices $(i,j)$ in model \eqref{eq:modelmc}. 
In this section we denote by $\hat \theta$ the following estimator
\begin{equation}\label{bic}
\hat \theta\in \underset{\theta=XBZ^T\in \mathcal{X}}{\arg \min}\; \left \{\Vert Y-\theta_{\Omega}\Vert^{2}_{2}+\lambda R(\theta)\right \}
\end{equation}
where $\lambda>0$ is a regularization parameter.  Note that this estimator does not require the knowledge of $p$.
The following theorem proved in Appendix \ref{proof_thm_mc} gives an upper bound on the estimation error of $\hat \theta$.
\begin{Theorem}\label{thm_mc_bic} 
Assume that  $n\,m\log (3\sqrt{n\wedge m})\geq 6\log \left (k_n\,k_m\right )$ and $d\geq 10$. Let $\lambda=8(\sigma\vee\Q)^{2}$.
% where $c$ is a sufficiently large absolute constant.
Then, for any $\theta^{*}\in \mathbb{R}^{n\times m}$, with $\mathbb{P}_{\theta^*}$-probability at least $1-5\exp(-d/10)-2\exp\left (-pnm\right )$ the estimator \eqref{bic} satisfies
\begin{align*}
&\Vert {\hat \theta}-\theta^{*} \Vert ^{2}_{2}\leq C\,
\underset{\theta\in\mathcal{X}}{\inf}\,\left \{\Vert { \theta}-\theta^{*}\Vert ^{2}_{2}+\frac{\left (\sigma\vee\Q\right )^{2}}{p}\,R(\theta)\right \}
\end{align*}
where $C>0$ is an absolute constant.
%and, in particular,
%\begin{align*}
%&\left \Vert {\hat \theta}-\theta^{*}\right \Vert ^{2}_{2}\leq \dfrac{\widetilde C}{p} \left \{k_nk_m+ns_n\log\left (e^{2}B_{\max}k_ns_m\right )+ms_m\log\left (e^{2}B_{\max}k_ms_n\right )\right \}.
%\end{align*}
%where $\widetilde C$ is a absolute constant that depends only on $\sigma,K$ and $\Q$
\end{Theorem}
Theorem \ref{thm_mc_bic} implies that for the general scheme of matrix completion and general alphabets we obtain the following upper bound which departs from the lower bound of Theorem \ref{thm:lowerbound} by a logarithmic factor: 
\begin{align*}
&\inf_{\theta\in\mathcal{X}}
\mathbb{P}_{\theta} \left(
\Vert {\hat \theta}-\theta \Vert^{2}_{2}\le C\,
\frac{\left  (\sigma\vee\Q\right )^{2}}{p} \left [\log(n\wedge m)+\log(s_ns_m)\right ] (R_X + R_B + R_Z ) \right)\\
& \ge
1-5\exp(-d/6)-2\exp\left (-pnm\right ),
\end{align*}
where $C>0$ is an absolute constant.

We finish this section by two remarks.

 1. {\it Structured matrix estimation.} In the case of complete observations, that is $p=1$, the estimator \eqref{bic} coincides with the following estimator
\begin{equation} \label{bic_estimation}
\hat \theta \in \underset{\theta=XBZ^T\in \mathcal{X}}{\arg \min}\; \left \{\Vert Y-\theta\Vert^{2}_{2}+\lambda R(\theta)\right \}.
\end{equation}
Then, one can show that, with high probability, the following upper bound on the estimation error holds
\begin{align*}
\Vert {\hat \theta}-\theta^{*} \Vert ^{2}_{2}\le C\, \sigma^{2} \left [\log(n\wedge m)+\log(s_ns_m)\right ]\left(R_X + R_B + R_Z \right).
\end{align*}
Here we do not need an upper bound on $\Vert \theta^{*}\Vert_{\infty}$.
At the same time, the estimator \eqref{bic_estimation} is adaptive to the sparsity parameter $(s_n,s_m)$.

%\item [] \textbf{Case of the finite dictionaries
%:} We do not need an upper bound on $B$ and we get
%\begin{align*}
%&\left \Vert {\hat \theta}-\theta^{*}\right \Vert ^{2}_{2}\leq \dfrac{\widetilde C}{p} \left \{k_nk_m+ns_n\log\left (ek_n/s_n\right )+ms_m\log\left (ek_m/s_m\right )\right \}.
%\end{align*}
%
%\item [] \textbf{Stochastic Block Model
%:} Note that in the case of the Stochastic Block Model ($k_n=k_m=k$, $s_n=s_m=1$ and $n=m$) our result coincides with previously obtained in \cite{gao2014rate} minimax optimal rate of convergence for this model:
%  \begin{align*}
%  &\left \Vert {\hat \theta}-\theta^{*}\right \Vert ^{2}_{2}\leq \widetilde C\left \{k^{2}+n\log\left (k\right )\right \}.
%  \end{align*}
  
 2. {\it
  Sparse Factor Model.}  Sparse Factor Model is studied in \cite{soni_factor}. With our notation, it corresponds to a particular case of $n=k_n$ and $X$ being the identity matrix with the difference that we consider row-sparse matrix $Z$ while $Z$ is assumed component-wise sparse in  \cite{soni_factor}.  Convergence rates obtained in \cite{soni_factor} are of the order $p^{-1}(nk_m+k_mm)$ (up to a logarithmic  factor). This is greater then the upper bound given by Theorem \ref{thm_mc_bic} which, in this setting, is of the order $p^{-1}\left [n(k_m\wedge m)+s_mm\right ]$.

  % % % % % % % % % % % % % % % % % % % % % % % % % % % % % % % % % % % % % % % % % % % % % % % % % % % % % % % % % %
  
  \section{Proofs}
% % % % % % % % % % % % % % % % % % % % % % % % % % % % % % % % % % % % % % % % % % % % % % % % % % % % % % % % % % % % % % %
%

% % % % % % % % % % % % % % % % % % % % % % % % % % % % % % % % % % % % % % % % % % % % % % % % %
% % % % % % % % % % % % % % % % % % % % % % % % % % % % % % % % % % % % % % % % % %
\appendix
% % % % % % % % % % % %
% % % % % % % % % % % % % % % % % % % % % % % % % % % % % % % % % % % % % % % % % % % % % % % % %
\section{Proof of Theorem \ref{thm:upperfinite}}\label{proof_thm_upperfinite}
%We denote by $\theta_0$ the best approximation of $\theta^*$ in the set $\Theta(s_n,s_m)$: $\theta_0 = \argmin_{\theta \in \Theta} \| \theta - \theta^* \|_2^2$. 
Set
$$
\bar R_1(n,m)=ns_n\log \left (\frac{ek_n|\mathcal{D}_\dimdone|}{s_n}\right )+r_nm,
$$
$$
\bar R_2(n,m)= ns_n\log \left (\frac{ek_n|\mathcal{D}_\dimdone|}{s_n}\right )+ms_m\log \left (\frac{ek_m|\mathcal{D}_\dimdtwo|}{s_m}\right )+r_nr_m.
$$
Since $\hat \theta$ is the least squares estimator on $\Theta(s_n,s_m)$, and $Y=\theta^{*}+W$, we have that for any $\bar\theta\in  \Theta(s_n,s_m)$,	 
	\begin{equation}\label{upperfinite_1}
	 \| \hat \theta-\theta^{*} \|^{2}_{2}\leq  \|  \bar\theta-\theta^{*} \|^{2}_{2}+2\langle \hat \theta-\bar\theta,W \rangle.
	\end{equation}
Now we use the following lemma proved in Section \ref{lemmas_theorem_upperfinite}. \begin{lemma}\label{lemma_brodten}
	Let $W\in \mathbb R^{n\times m}$ be a random matrix with independent $\sigma-$sub-Gaussian entries. Introduce the notation
$$
U_{\bar\theta}^*=\underset{\theta\in \Theta(s_n,s_m),\theta\not=\bar\theta}{\sup} \dfrac{\langle  \theta-\bar\theta,W\rangle^{2}}{\Vert  \theta-\bar\theta\Vert^{2}_{2}}.
$$
For any $t>0$, the following inequalities hold, where $C>0$ is an absolute constant:
	\begin{itemize}
		\item [(i)]
		$$%\begin{equation*}
		\hskip - 1 cm\underset{\bar\theta\in \Theta(s_n,s_m)}{\sup}	\mathbb{P}\left \{U_{\bar\theta}^*\geq 3\sigma^{2}\left (\bar R_1(n,m)+t\right )\right \}\leq e^{-t}, \quad \underset{\bar\theta\in \Theta(s_n,s_m)}{\sup}	\mathbb{E} (U_{\bar\theta}^* )\leq  C\sigma^{2}\bar R_1(n,m),
	$$
	%\end{equation*} 
	%
	\item [(ii)] 
	$$
	\hskip - 1 cm\underset{\bar\theta\in \Theta(s_n,s_m)}{\sup}	\mathbb{P}\left \{U_{\bar\theta}^*\geq 3\sigma^{2}\left (\bar R_2(n,m)+t\right )\right \}\leq e^{-t}, \quad \underset{\bar\theta\in \Theta(s_n,s_m)}{\sup}	\mathbb{E} (U_{\bar\theta}^* )\leq  C\sigma^{2}\bar R_2(n,m),
	$$
			\item [(iii)] 
			$$
	\hskip - 1 cm\underset{\bar\theta\in \Theta(s_n,s_m)}{\sup}	\mathbb{P}\left \{U_{\bar\theta}^*\geq 3\sigma^{2}\left (nm+t\right )\right \}\leq e^{-t}, \quad \underset{\bar\theta\in \Theta(s_n,s_m)}{\sup}	\mathbb{E} (U_{\bar\theta}^* )\leq  C\sigma^{2}nm.
	$$
	\end{itemize}	
\end{lemma}
Applying Lemma \ref{lemma_brodten} (i) to \eqref{upperfinite_1}, for any $\epsilon>0$, we get that
\begin{equation}\label{upperfinite_2}
\E \left\{  \Vert\hat{\mtheta} - \mtheta^{*} \Vert_2^2 \right\} \leq  (1+\epsilon)\left \Vert\mtheta^{*} - \bar\mtheta\right \Vert_2^2+\dfrac{C\sigma^{2}}{\epsilon}\bar R_1(n,m).
\end{equation}
On the other hand, applying Lemma \ref{lemma_brodten} (i) to $ \langle \hat \theta-\bar\theta,W \rangle= \langle  (\hat \theta-\bar\theta )^T,W^T \rangle$ with $n$ replaced by $m$, for all $\epsilon>0$, we get 
\begin{equation}\label{upperfinite_3}
\E \left\{ \Vert\hat{\mtheta} - \mtheta^{*} \Vert_2^2 \right\} \leq  (1+\epsilon) \Vert\mtheta^{*} - \bar\mtheta \Vert_2^2+\dfrac{C\sigma^{2}}{\epsilon}\bar R_1(m,n).
\end{equation}
Finally using  Lemma \ref{lemma_brodten} (ii) and (iii) we get
%$\mTheta(s_m,s_n)\subset \mathbb R^{n\times m}$ and Lemma  \ref{lem:kakade} we get
\begin{equation}\label{upperfinite_4}
\E \left\{  \Vert\hat{\mtheta} - \mtheta^{*} \Vert_2^2 \right\}  \leq  (1+\epsilon) \Vert\mtheta^{*}- \bar\mtheta \Vert_2^2+\dfrac{C\sigma^{2}}{\epsilon}\bar R_2(n,m)
\end{equation}
and 
\begin{equation}\label{upperfinite_5}
\E \left\{\Vert\hat{\mtheta} - \mtheta^{*} \Vert_2^2 \right\}  \leq  (1+\epsilon) \Vert\mtheta^{*}- \bar\mtheta \Vert_2^2+\dfrac{C\sigma^{2}}{\epsilon} nm.
\end{equation}
Inequalities \eqref{upperfinite_2} - \eqref{upperfinite_5} imply that for all $\epsilon>0$ and all $\bar{\theta}\in \Theta(s_n,s_m)$
\begin{equation*}
\begin{split}
\E \left\{ \Vert\hat{\mtheta} - \mtheta^{*} \Vert_2^2 \right\} &\leq  (1+\epsilon) \Vert\mtheta^{*} - \bar\mtheta \Vert_2^2+\dfrac{C\sigma^{2}}{\epsilon}\bar R_3(n,m).
%\left \{ns_n\log \left (\frac{ek_n|\mathcal{D}_\dimdone|}{s_n}\right )+r_nm,\right .\\ & \hskip -2 cm \left. ms_m\log \left (\frac{ek_m|\mathcal{D}_\dimdtwo|}{s_m}\right )+r_mn,ns_n\log \left (\frac{ek_n|\mathcal{D}_\dimdone|}{s_n}\right )+ms_m\log \left (\frac{ek_m|\mathcal{D}_\dimdtwo|}{s_m}\right )+r_nr_m\right \}.
\end{split}
\end{equation*}
where $\bar R_3(n,m)=\min\{\bar R_1(n,m),\bar R_1(m,n),\bar R_2(n,m),nm\}$.
Taking the supremum over  $\bar{\theta}\in \Theta(s_n,s_m)$ and simplifying the expression for $\bar R_3(n,m)$ we obtain the result of Theorem \ref{thm:upperfinite}.

   \section{Proof of Lower Bounds}
   
   \subsection{Proof of Theorem \ref{thm:lowerbound}} \label{sec:prooflower}
   %\begin{proof} [Proofs of the Theorem \ref{thm:lowerbound}]
   \noindent \textit{Lower bound with the terms $R_X$ and $R_Z$.} 
   We only prove the lower bound with the term $R_Z$ by fixing $X=X_0$ and $B=B_0$, where $X_0$ and $B_0$ are matrices specified below. The bound with $R_X$ is analogous.  Fix
   \[
       X_0= 
   \begin{cases}
       [\mathbf{I}_{k_n \times k_n}, \mathbf{0}]^T,& \text{if~} n \geq k_n,\\
       [\mathbf{I}_{n \times n}, \mathbf{0}],              & \text{otherwise}.
   \end{cases}
   \]
   By Lemma~\ref{lm:select}, for $k_m\ge 2$, we can find $S_0 \subseteq \{0,1\}^{k_m}$ with the following properties: 
   \begin{itemize}
   \item[(i)] $\log |S_0|\ge c_1^* s_m \log \frac{ek_m}{s_m}$,
   \item[(ii)] $c_2^*s_m \le \|a\|_0 \le s_m$ for all $a \in S_0$, and $\|a\|_0 = s_m$ for all $a \in S_0$ if $s_m\le k_m/2$,
   \item[(iii)] $\|a-b\|_2^2 \ge  c_3^* s_m$ for all $a, b \in S_0$ such that $a\ne b$,
   \end{itemize}
where $c_j^*>0$, $j=1,2,3$, are absolute constants.

Assume first that $k_m\ge 2$ and  $\min\{ \frac{r_n}{97},c_1^* s_m \log \frac{ek_m}{s_m}\}\ge \log 8$. 
         Then, choose an arbitrary subset $\mathcal{S} \subseteq S_0$ of cardinality $|\mathcal{S}| = \lfloor\exp\left (\min\{ \frac{r_n}{97}, c_1^* s_m \log \frac{ek_m}{s_m}\}\right )\rfloor$ where $r_n=n \wedge k_n$ and we denote by $\lfloor x\rfloor$  the integer part of $x$. Since $\log |\mathcal{S}| \le r_n/96$, Lemma \ref{lm:specialB} implies that there exists a matrix $Q \in \{-1,1\}^{r_n \times k_m}$ such that, for any $a,b \in \mathcal{S}$, 
   \begin{equation} \label{eq:propQ}
   \frac{r_n}{2} \|a - b\|_2^2 \le \| Qa - Qb\|_2^2 \le \frac{3r_n}{2} \|a- b\|_2^2.
   \end{equation}
   For this $Q$, let 
   $$
   B_0 = [\delta Q, \mathbf{0}_{(k_n-r_n) \times k_m}]^T
   $$ 
   with $\delta>0$ to be specified below. Define $\mathcal{Z} = \{Z \in \{0,1\}^{m \times k_m}, Z_{i\cdot} \in \mathcal{S} \text{ for all } i \in [m]\}$ and $T_Z = \left\{\theta=X_0 B_0 Z^T, Z \in \mathcal{Z}  \right\}$. We have $T_Z \subseteq \Theta (s_n,s_m)$ and $\log |T_Z| = \log |\mathcal{Z}| = m \log |\mathcal{S}|$. 
  
   For any matrices $\theta=X_0 B_0 Z^T\in T_Z$ , and $\bar\theta=X_0 B_0 \bar{Z}^T\in T_Z $ we have 
   \begin{eqnarray*}
   ||\theta - \bar\theta||_2^2 = \delta^2 ||QZ^T - Q\bar{Z}^T||_2^2 = \delta^2 \sum_{i=1}^{m} ||QZ_{i\cdot}^T - Q\bar{Z}_{i\cdot}^T||_2^2.
   \end{eqnarray*}
   Using (\ref{eq:propQ}) and property (iii) of $S_0$, we find
   \begin{equation} \label{eq:packZ}
   ||\theta - \bar{\theta}||_2^2 \ge \frac{r_n \delta^2}{2} \sum_{i=1}^{m} \|Z_{i\cdot} - \bar{Z}_{i\cdot}\|_2^2 \ge \frac{c_3^*r_n \delta^2 m s_m}{2}.
   \end{equation}
   On the other hand, Lemma \ref{lm:KLchi} together with (\ref{eq:propQ}) implies that the Kullback-Leibler divergence between $\mathbb{P}_{\theta}$ and $ \mathbb{P}_{\bar{\theta}}$ satisfies
   \begin{equation} \label{eq:KLZ}
   KL(\mathbb{P}_{\theta}, \mathbb{P}_{\bar{\theta}}) = \frac{p}{2\sigma^2} ||\theta - \bar{\theta}||_2^2 \le \frac{3p r_n \delta^2}{4 \sigma^2} \sum_{i=1}^{m} \|Z_{i\cdot} - \bar{Z}_{i\cdot}\|_2^2 \le \frac{3p r_n \delta^2 m s_m }{2 \sigma^2}.
   \end{equation}
   If we choose now $\delta^2=\frac{C_0 \sigma^2}{pr_ns_m}  \log |\mathcal{S}|$ for some absolute constant $C_0>0$ small enough, then (\ref{eq:packZ}), (\ref{eq:KLZ}), Theorem 2.5 in \cite{tsy_09} and the fact that $|\mathcal{S}|\ge 8$ imply that
   \begin{equation} \label{eq:lowerterm3}
   \inf_{\hat{\vartheta}}\sup_{\theta\in T_Z}\mathbb{P}_\theta\left\{||\hat{\vartheta} - \theta||_2^2 \geq C_1 \frac{\sigma^2}{p} \left(mr_n \wedge m s_m\log \frac{ek_m}{s_m} \right) \right\}\geq 0.7
   \end{equation}
   for some absolute constant $C_1>0$. This yields the term of the lower bound containing $R_Z$ in the case when  $k_m\ge 2$ and  $\min\{ \frac{r_n}{97},c_1^* s_m \log \frac{ek_m}{s_m}\}\ge \log 8$. In the complementary case, when $k_m=1$ or $\min\{ \frac{r_n}{97},c_1^* s_m \log \frac{ek_m}{s_m}\}< \log 8$, the value $R_Z$ is smaller than $C\sigma^2m/p$ for an absolute constant $C>0$. Thus, in this case, it suffices to prove the lower bound of order $m/p$. To do this, let the matrices $X_0$ and $B_0$ be such that their $(1,1)$th entry is equal to 1 and all other entries are 0, and consider the set of matrices $Z$ such that their first column is a binary vector in $\{0,1\}^m$ and all other columns are 0. This defines a set of matrices $\theta=X_0B_0Z^T$ contained in $\Theta (s_n,s_m)$, which is isometric, under the Frobenius norm, to the set of binary vectors $\{0,1\}^m$ equipped with the Euclidean norm. Therefore, a lower bound of order $m/p$
follows in a standard way as for vector estimation problem. We omit further details.
 
   Analogously, by permuting $n$ and $m$, we obtain that 
   \begin{equation} \label{eq:lowerterm1}
   \inf_{\hat{\vartheta}}\sup_{\theta\in \Theta (s_n,s_m)}\mathbb{P}_\theta\left\{ ||\hat{\vartheta} - \theta||_2^2 \geq C_1 \frac{\sigma^2}{p} \left(nr_m \wedge n s_n\log \frac{ek_n}{s_n} \right) \right\}\geq 0.7,
   \end{equation}
   which yields the term of the lower bound containing $R_X$.
   
   \smallskip
   
   \noindent \textit{Lower bound with the term $R_B$.} To obtain the term containing $R_B$ in the lower bound (\ref{eq:lower}), we fix $X=X_0$ and $Z=Z_0$ where
   \[
       X_0= 
   \begin{cases}
       [\mathbf{I}_{k_n \times k_n}, \mathbf{0}]^T,& \text{if~} n \geq k_n,\\
       [\mathbf{I}_{n \times n}, \mathbf{0}],              & \text{otherwise,}
   \end{cases}
   \quad\text{and}\quad
       Z_0= 
   \begin{cases}
       [\mathbf{I}_{k_m \times k_m}, \mathbf{0}]^T,& \text{if~} m \geq k_m,\\
       [\mathbf{I}_{m \times m}, \mathbf{0}],              & \text{otherwise}.
   \end{cases}
   \]
   We first note that if $r_nr_m< 16$, the lower bound with term $R_B$ is trivially obtained by distinguishing between two matrices. For $r_nr_m\ge 16$, by vectorizing a $r_n \times r_m$ matrix into a $r_nr_m$ dimensional vector, and applying the Varshamov-Gilbert bound \cite[Lemma 2.9]{tsy_09} we obtain that there exists a subset $\mathcal{B} \subseteq \{0,1\}^{r_n \times r_m}$  such that for any $Q, \bar{Q} \in \mathcal{B}$,
   \begin{equation*}
   || Q - \bar{Q} ||_2^2 = \sum_{i,j} \mathbf{1}\{Q_{ij} \neq \bar{Q}_{ij}\} \ge \frac{r_n r_m}{8}
   \end{equation*}
   and $\log |\mathcal{B}| \ge \frac{r_nr_m}{8} $. We
   define
   $$T_B = \left\{ \theta = X_0 B Z_0^T, B =  \delta \begin{bmatrix} Q & \mathbf{0} \\ \mathbf{0} & \mathbf{0} \end{bmatrix},  Q \in \mathcal{B} \right\}. $$
   Clearly, $T_B\subseteq \Theta (s_n,s_m)$. For any $\theta=X_0 B Z_0^T\in T_B, \bar{\theta}=X_0 \bar{B} Z_0^T \in T_B$, we have
   \begin{equation*}
   ||\theta - \bar{\theta}||_2^2 = ||B - \bar{B}||_2^2 = \delta^2  ||\mQ - \bar{\mQ}||_2^2 \ge \frac{r_nr_m \delta^2}{8}.
   \end{equation*}
   %Since $Q \in \{0,1\}^{k_n \times k_m}$,
    Lemma \ref{lm:KLchi} implies that
   $KL(\mathbb{P}_{\theta}, \mathbb{P}_{\bar{\theta}}) =  \frac{p}{2\sigma^2} ||\theta - \bar{\theta}||_2^2 \le  \frac{p \delta^2 r_n r_m}{2\sigma^2}.$
   Choosing $\delta^2 = C_0' \sigma^2/p$ for some constant $C_0'>0$ small enough and using Theorem 2.5 in \cite{tsy_09} we obtain
   \begin{equation} \label{eq:lowerterm2}
   \inf_{\hat{\vartheta}}\sup_{\theta\in T_B}\mathbb{P}\left\{\|\hat{\vartheta} - \theta\|_2^2 \geq C_2 \frac{\sigma^2}{p} r_n r_m \right\}\geq 0.7
   \end{equation}
   for some absolute constant $C_2>0$. 
   
   \medskip
  Combining  (\ref{eq:lowerterm3}), (\ref{eq:lowerterm1}) and (\ref{eq:lowerterm2}) proves the lower bound \eqref{eq:lower}. The bound \eqref{eq:expectlower} follows from \eqref{eq:lower} and Markov's inequality. 
  
  Finally, the lower bounds for the classes $\Theta_*(s_n,s_m)$ with $s_n, s_m\in \{0,1\}$ are proved analogously. It suffices to note that, if $s_m=0$, there is no matrix $Z$ in the definition of the class and thus there is no term $R_Z$.  If $s_m=1$
 we follow the above argument corresponding to $R_Z$ with the only difference that, by (ii) of Lemma~\ref{lm:select}, we can grant the exact equality  $\|a\|_0 = 1$ for all $a \in S_0$ and  $k_m\ge 2$. We omit further details.
   %Olga:ok
   %\end{proof}
   % % % % % % % % % % % % % % % % % % % % % % % % % % % % % % % % % % % % % % % % % % % % % % % % % % % % % % % % % % % % % % % %
   
   \subsection{Proof of Corollary \ref{cor:speclower}} \label{sec:proofspeclower}
   %\begin{proof}
   Define
   \begin{align*}
\Gamma_2=\{\theta=XBZ^T\,:\,X\in \mathcal{A}_{s_n}(n,2),B\in \mathbb{R}^{2\times 2}\;\text{and}\;Z\in \mathcal{A}_{s_m}(m,2) \}.
\end{align*}  
For any $\theta = XBZ^T \in \Gamma_2$, let 
  $$
  \tilde{X} = [X, \mathbf{0}_{n \times (k_n-2)}], \tilde{Z} = [Z, \mathbf{0}_{m \times (k_m-2)}], \tilde{B} = \begin{bmatrix} B & \mathbf{0}_{2 \times (k_m-2)} \\ \mathbf{0}_{(k_n-2) \times 2} & \mathbf{0} \end{bmatrix}.
  $$ 
   We have $\theta = XBZ^T = \tilde{X} \tilde{B} \tilde{Z}^T\in \Theta(s_n,s_m)$, which implies that $\Gamma_2 \subseteq \Theta(s_n,s_m)$. Thus, for any $t>0$, and any estimator $\hat{\vartheta}\in \R^{n\times m}$ we have
   \begin{equation} \label{eq:spec1} 
   \sup_{\theta \in \Theta(s_n,s_m)} \mathbb{P}_\theta (\|\hat{\vartheta} - \theta\|^2 \ge t) \ge \sup_{\theta \in \Gamma_2} \mathbb{P}_\theta (\|\hat{\vartheta} - \theta\|^2 \ge t).
   \end{equation}
   For an estimator $\hat \vartheta\in \R^{n\times m}$, let $\hat{\vartheta}_2\in \R^{n\times m} $ be the closest  matrix to $\hat \vartheta$ in the Frobenius norm among all matrices of rank at most $2$. Since $\theta \in \Gamma_2$ is of rank at most $2$, we have $\|\hat{\vartheta} - \hat{\vartheta}_2\| \le \|\hat{\vartheta} - \theta\|$ and
   $$ 
   \|\theta - \hat{\vartheta}_2\|_2^2 \le 4 \|\theta - \hat{\vartheta}_2 \|^2 \le 8  \left( \|\theta - \hat{\vartheta}\|^2 + \|\hat{\vartheta}-\hat{\vartheta}_2\|^2 \right) \le 16 \|\theta - \hat{\vartheta}\|^2. 
   $$ 
   Thus,
   $$ \mathbb{P}_\theta (\|\hat{\vartheta} - \theta\|^2 \ge t) \ge \mathbb{P}_\theta (\|\hat{\vartheta}_2 - \theta\|_2^2 \ge 16t) $$
   for any $\theta \in \Gamma_2$ and any estimator $\hat{\vartheta}$. The last inequality and (\ref{eq:spec1}) imply 
   \begin{eqnarray*} 
   \inf_{\hat{\vartheta}} \sup_{\theta \in \Theta(s_n,s_m)} \mathbb{P}_\theta (\|\hat{\vartheta} - \theta\|^2 \ge t) &\ge&  \inf_{\hat{\vartheta}}\sup_{\theta \in \Gamma_2} \mathbb{P}_\theta (\|\hat{\vartheta}- \theta\|_2^2 \ge 16t). %\\ &\ge&
    %\inf_{\tilde{\theta}} \sup_{\theta \in \Gamma_2} \mathbb{P}_\theta (\|\tilde{\theta} - \theta\|_2^2 \ge 16t).
    \end{eqnarray*}
    %where $\inf_{\hat{\vartheta}'}$ denotes the infimum over all estimators with values in $\R^{n\times m}$.
   The result of Corollary \ref{cor:speclower} follows now by choosing $ t = C_3 \frac{\sigma^2}{p}(n + m)$ for some constant $C_3>0$ and using Theorem \ref{thm:lowerbound} with $k_m=k_n=2$. 
   %Olga:ok
   % % % % % % % % % % % % % % % % % % % % % % % % % % % % % % % % % % % % % 
   
   \section{Proof of Theorem \ref{corollary_spectral}}\label{proof_corollary_spectral}
   Note that $Y'-\theta^{*} = Y/p -\theta^{*} = W$. It is straightforward to see that if $\lambda\geq \Vert W \Vert$, then 
    \begin{equation*}
    \Vert \tilde \theta-\theta^{*}\Vert\leq 2\lambda.
    \end{equation*}
   Thus, to prove the theorem, it suffices to show that for $\lambda=c(b+\Q)\sqrt{\frac{n\vee m}{p}}$ with $c>0$ large enough we
   have $\lambda\geq \Vert W\Vert$ with high probability.

   %Let $\Delta_{ij}=e_{i}(n)e^T_{j}(m)$ where $e_{k}(l)$ are the canonical basis vectors in $\mathbb R^l$. 
   Since $W_{ij}= \theta_{ij}(E_{ij}-p)/p +\xi_{ij}E_{ij}/p$, we have
       \begin{align}\label{spectral_1}
    \Vert W \Vert\leq p^{-1}(\left \Vert \Sigma_1 \right \Vert +\left \Vert \Sigma_2 \right \Vert)
    \end{align}
    where
    $
     \Sigma_1 \in \R^{n\times m}
    $
    is a matrix with entries $\xi_{ij}E_{ij}$ and $\Sigma_2\in \R^{n\times m}$ is a matrix with entries $\theta_{ij}(E_{ij}-p)$.
    The second term in \eqref{spectral_1} is controlled using the following bound on the spectral norms of random matrices.
      \begin{proposition}[\cite{Bandeira}]\label{pr1}
      Let $A$ be an $n\times m$ matrix whose entries $A_{ij}$ are independent centered bounded random variables. Then, for any $0<\epsilon\leq 1/2$ there exists an absolute constant $c_{\epsilon}$ depending only on $\epsilon$ such that, for every $t> 0$,
      $$ \mathbb P\left \{ \left \Vert A\right\Vert\geq (1+\epsilon)2\sqrt{2}(\sigma_1\vee \sigma_2)+t\right \}\leq (n\wedge m)\exp\left (-\frac{t^{2}}{c_{\epsilon}\sigma^{2}_*}\right )$$ 
      where 
      $$\sigma_{1}=\underset{i}{\max}\sqrt{\sum_{j}\mathbb E[A_{ij}^{2}]},
      \quad
      \sigma_{2}=\underset{j}{\max}\sqrt{\sum_{i}\mathbb E[A_{ij}^{2}]},
        \quad \sigma_{*}=\underset{ij}{\max}\vert A_{ij}\vert.$$
         \end{proposition}
   We now apply Proposition \ref{pr1} with   $ A_{ij} =(E_{ij}-p)\theta_{ij}$. Then 
         \begin{equation*}
         \begin{split}
         \sigma_1\leq \Q\sqrt{np},\quad  \sigma_2\leq  \Q\sqrt{mp}\quad \text{and}\quad\sigma_{*}\leq \Q.
         \end{split}
         \end{equation*}
  Using these bounds and  taking in Proposition \ref{pr1} the values $\epsilon=1/2$ and $t=\sqrt{c_{1/2}}\Q\log(n+m)$ we obtain that there exists an absolute constant $c^{*}>0$ such that  
      \begin{align*}
      \left \Vert \Sigma_2 \right \Vert\leq 3\Q\sqrt{2(n\vee m)p}+c^{*}\Q\,\sqrt{2\log(n+m)}
      \end{align*} with probability at least $1-1/(n+m)$. 
   Similarly, there exists an absolute constants $c^{*}>0$ such that,  with probability at least $1-1/(n+m)$, 
      \begin{align*}
      \left \Vert \Sigma_1\right \Vert\leq 3\sigma\sqrt{2(n\vee m)p}+c^{*}b\,\sqrt{2\log(n+m)}.
      \end{align*}
     Using these remarks, the assumption $p\geq \log(n+m)/(n\vee m)$, and \eqref{spectral_1} we obtain that the choice $\lambda=c(b+\Q)\sqrt{\frac{n\vee m}{p}}$ with $c>0$ large enough implies the inequality $\lambda\geq \Vert W\Vert$  with 
     probability at least $1-2/(n+m)$.
      % % % % % % % % % % % % % % % % % % % % % % % % % % % % % 
  % % % % % % % % % % % % % % % % % % % % % % % % % % % % % % % % % % % % % % % 
  \section{Proof of Theorem \ref{thm:oracleupper}}\label{proof_thm:oracleupper}  
  Let $R=\sqrt{\frac{3(\Q^2+\sigma^2)}{p} N \epsilon_0^2}$. We consider two cases separately. 
  \paragraph{Case 1:} $\hat{\theta} \in \mathcal{G} \triangleq \{\theta \in \Theta, \|\theta - \theta_0\|_2 \le 2^{s^{*}}R\}$. Then  the desired result follows from the fact that $\|\hat{\theta}-\theta_0\|_2 \le 2^{s^{*}}R$ and from the inequality
  $$ \|\hat{\theta} - \theta^* \|_2^2 \le 2 \|\theta_0-\theta^*\|_2^2 + 2\|\hat{\theta} - \theta_0\|_2^2. $$
  \paragraph{Case 2:} $\hat{\theta} \not\in \mathcal{G}$. 
  The definition of the least squares estimator (\ref{eq:ls}) implies 
  %\begin{equation}\label{eq:basic}
  $\| Y' - \hat{\theta} \|_2^2 \le \| Y' - \theta_0 \|_2^2.$
  %\end{equation}
  Writing $Y'$ as $\theta^* + W$ and rearranging, we obtain
  $$ \|\hat{\theta} - \theta^* \|_2^2 \le \|\theta_0-\theta^*\|_2^2 + 2 \langle \hat{\theta} - \theta_0, W \rangle. $$
  Since $\hat{\theta} \in \mathcal{G}^c$, Lemma \ref{lm:chaining} yields
  \begin{equation} \label{eq:chaining}
  \langle \hat{\theta} - \theta_0, W \rangle \le \frac{1}{8} \|\theta - \theta_0\|_2^2 + 96 R^2
  \end{equation}
  with probability greater than $1-4\exp(-\alpha R^2/2) - 2\exp(-pN/6)$ where $\alpha=\frac{p}{6(\Q^{2}+\sigma^{2})}$. On the event where (\ref{eq:chaining}) holds, 
  \begin{eqnarray*}
  \|\hat{\theta} - \theta^* \|_2^2 &\le& \|\theta_0-\theta^*\|_2^2 + \frac{1}{4} \|\theta - \theta_0\|_2^2 + 192 R^2 \\
  &\le&  \frac{3}{2}\|\theta_0-\theta^*\|_2^2 + \frac{1}{2} \|\hat{\theta} - \theta^*\|_2^2 + 192 R^2.
  \end{eqnarray*}
  This yields $\|\hat{\theta} - \theta^* \|_2^2 \le 3\|\theta_0-\theta^*\|_2^2 + 384 R^2$ with probability greater than $1-4\exp(-\alpha R^2/2) - \exp(-pN/6)$.
  %Olga:ok
% % % % % % % % % % % % % % % % % % % % % % % % % % % % % % % % % 
% % % % % % % % % % % % % % % % % % %
\section{Proof of Proposition \ref{prop:coveringbound}}\label{proof_prop:coveringbound}
%  This section is devoted to get the $\epsilon$-covering number for the parameter space $\Theta_u(s_n,s_m)$.
  %\begin{proof}[Proof of Proposition \ref{prop:coveringbound}]
  We start by proving the upper bound corresponding to $R_1(\epsilon)$. Let $\bar{\mathcal{A}}_{n}(\delta)$ denote the $\delta$-covering set of $\widetilde{\mathcal{A}}_{n}$   %and $\bar{\mathcal{A}}_{m}(\delta_2)$ be the $\delta_2$-covering set of $\mathcal{A}_{m}$ 
  under the $\ell_{\infty}$ norm. 
  For any given $\theta=XBZ^T \in \widetilde\Theta_{u}(s_n,s_m)$, there exist
  $X_0 \in \bar{\mathcal{A}}_{n}(\delta_1)$ and $Z_0 \in \bar{\mathcal{A}}_{m}(\delta_2)$ such that $||X - X_0||_\infty \le  \delta_1$ and $||Z - Z_0||_\infty \le  \delta_2$. For such $X_0$ and $Z_0$, let $\overline{\mathcal{T}}(X_0, Z_0)$ be the $\epsilon/3$-covering set of $\mathcal{T}_u(X_0,Z_0)$ defined in Lemma \ref{lm:coveringB}. For any $B\in \mathcal{T}_u(X_0,Z_0)$, there exists $B_0\in \overline{\mathcal{T}}(X_0, Z_0)$ such that
  \begin{eqnarray*}
  && ||X B Z^T -  X_0 B_0 Z_0^T ||_2 \\ &\le& ||X B Z^T-  X_0 B Z^T||_2 + 
  ||X_0 B Z^T - X_0 B Z_0^T||_2 + || X_0  B Z_0^T - X_0 B_0 Z_0^T||_2 \\
  &\le& 2s_n  \sqrt{nm} \|X - X_0\|_\infty \|B Z\|_\infty+ 2s_m \sqrt{nm}\|X_0B\|_\infty \|Z-Z_0\|_\infty + \frac{\epsilon}{3}\\
  &\le& 2 B_{max} \sqrt{nm} s_n s_m (\delta_1+\delta_2) + \frac{\epsilon}{3},
  \end{eqnarray*}
  where in the second inequality we have used Lemma \ref{lm:matrixbound} and the last inequality is due to the assumptions that $||B||_\infty \le B_{max}, ||X||_{\infty} \le 1$ and $||Z||_{\infty} \le 1$. Choosing $\delta_1=\delta_2 = \epsilon/(6B_{max}\sqrt{ nm} s_n s_m)$ we get that the set
  $$\overline{\Theta}_u := \bigcup_{X_0 \in \overline{\mathcal{A}}_{n}(\delta_1), Z_0 \in \overline{\mathcal{A}}_{m}(\delta_2)} \overline{\mathcal{T}}(X_0,Z_0)$$ 
  is an $\epsilon$-covering set of $\widetilde\Theta_{u}(s_n,s_m)$. Then, Lemmas \ref{lm:coveringB} and \ref{lm:coveringZ} imply
  \begin{eqnarray*}
  \log \mathcal{N}_\epsilon\left ( \widetilde\Theta_{u}(s_n,s_m)\right )&\le& \log \left| \overline{\mathcal{A}}_{n}(\delta_1) \right| + \log \left| \overline{\mathcal{A}}_{m}(\delta_2) \right| + \max_{X_0, Z_0}\log \left| \overline{\mathcal{T}}(X_0,Z_0) \right| \\
  &\le& ns_n \log \frac{ek_n}{s_n} + ns_n \log \frac{6B_{max} \sqrt{mn} s_m s_n}{\epsilon} + r_n r_m \log \frac{9u}{\epsilon} \\
  && +  ms_m \log \frac{ek_m}{s_m} + ms_m \log \frac{6B_{max} \sqrt{mn} s_m s_n}{\epsilon}.
  \end{eqnarray*}
  To get upper bounds corresponding to $R_2(\epsilon),R_3(\epsilon)$ and $R_4(\epsilon)$ we define 
  $$\Theta_u^{1} = \left\{ \theta= A Z^T, A \in [-s_nB_{max}, s_nB_{max}]^{n \times k_m}, Z \in \widetilde{\mathcal{A}}_m, \| \theta - \theta_0\|_2 \le u \right\},$$
  $$\Theta_u^{2} = \left\{ \theta= X G^T, X \in \widetilde{\mathcal{A}}_n, G \in [-s_mB_{max}, s_mB_{max}]^{ k_n \times m}, \| \theta - \theta_0\|_2 \le u \right\},$$
  and
  $$\Theta_u^{3} = \left\{ \theta \in [-s_ns_mB_{max}, s_ns_mB_{max}]^{n \times m}, \| \theta - \theta_0\|_2 \le u \right\}.$$
  It is easy to verify that $\widetilde\Theta_{u}(s_n,s_m) \subseteq \Theta_u^{1}, \widetilde\Theta_{u}(s_n,s_m)\subseteq \Theta_u^2$ and $\widetilde\Theta_{u}(s_n,s_m) \subseteq \Theta_u^3$. Using the same techniques as above we obtain
  \begin{eqnarray*}
  \log \mathcal{N}_\epsilon\left ( \widetilde\Theta_{u}(s_n,s_m)\right )\le nr_m \log \frac{6u}{\epsilon} + ms_m \log \frac{ek_m}{s_m} + ms_m \log \frac{2B_{max}\sqrt{mn}s_ms_n}{\epsilon},
  \end{eqnarray*}
  \begin{eqnarray*}
  \log \mathcal{N}_\epsilon\left ( \widetilde\Theta_{u}(s_n,s_m)\right )&\le& mr_n \log \frac{6u}{\epsilon} + ns_n \log \frac{ek_n}{s_n} + ns_n \log \frac{2B_{max}\sqrt{mn}s_ms_n}{\epsilon},
  \end{eqnarray*}
  and
  \begin{eqnarray*}
  \log \mathcal{N}_\epsilon\left ( \widetilde\Theta_{u}(s_n,s_m)\right )&\le& m n \log \frac{3u}{\epsilon}.\end{eqnarray*}
  Combining these bounds completes the proof of Proposition \ref{prop:coveringbound}.
  %Olga:ok
  %\end{proof}
  
% % % % % % % % % % % % % % % % % % % % % % % % % % % % % % % % % % % % % % % % % % % % % % % % % % % 
% % % % % % % % % % % % % % % % % % % % % % % % % % % % % % % % % % % %

\section{Proof of Theorem \ref{thm_mc_bic}}\label{proof_thm_mc}
Let $\mathcal{I}=\underset{(s_n,s_m)}{\min} R(s_n,s_m)$ and $\nu^{2}=\frac{\left (\sigma\vee\Q\right )^{2}}{p}\mathcal{I}$. By the definition of $ R(s_n,s_m)$ in \eqref{def_R_s} we have $\mathcal{I}\geq d$. Note first  that if $\Vert \hat \theta-\theta^{*}\Vert_{2}\leq 2^{7}\nu,$
then Theorem \ref{thm_mc_bic} holds trivially. So, without loss of generality, we can assume that 
$\hat \theta\in \mathcal{X}_{\nu}\triangleq \{\theta \in \mathcal{X}: \|\theta - \theta^{*}\|_2 > 2^{7}\nu\}$.
By the definition \eqref{bic} of the estimator $\hat \theta=\hat X\hat B\hat Z^T$ we have that for any $\theta=X B Z^T\in \mathcal{X}$ 
\begin{align*} %\label{thm1_1}
\Vert Y-\hat \theta_{\Omega}\Vert ^{2}_{2}+\lambda R(\hat\theta)&\leq\Vert Y-\theta_{\Omega}\Vert ^{2}_{2}+\lambda R(\theta)
   \end{align*}  
which implies
\begin{align} \label{thm_bic_1}
 \Vert \hat \theta_{\Omega}-\theta^{*}_{\Omega}\Vert ^{2}_{2}&\leq \Vert \theta_{\Omega}-\theta^{*}_{\Omega}\Vert ^{2}_{2}-2\langle \xi_{\Omega},\theta-\theta^{*}\rangle+\lambda R(\theta)+2\langle \xi_{\Omega},\hat\theta-\theta^{*}\rangle-\lambda R(\hat\theta)
   \end{align}
where we set $\xi=(\xi_{ij})$. We will bound each term in \eqref{thm_bic_1} separately.   
 Lemma \ref{stoch} implies that with probability at least $1-\exp(-pnm)-2\exp(-d/10)$,
 \begin{align}\label{thm_bic_2}
 \langle \xi_{\Omega},\hat\theta-\theta^{*}\rangle\leq 2 \left (\sigma\vee\Q\right )^{2}R(\hat \theta)+\frac{p}{8}\Vert \hat \theta-\theta^{*}\Vert ^{2}_{2}.
 \end{align}
To control   $\langle \xi_{\Omega},\theta-\theta^{*}\rangle$, we use Lemma \ref{lemma_bernstein_1} with $t=p\Vert \theta-\theta^{*}\Vert ^{2}_{2}+\left (\sigma\vee\Q\right )^{2}R(\theta)$. It follows  that, with probability at least $1-\exp(-d/2)$, 
 \begin{align}\label{thm_bic_4}
 \langle \xi_{\Omega},\theta-\theta^{*}\rangle\leq \left (\sigma\vee\Q\right )^{2}R( \theta)+p\Vert  \theta-\theta^{*}\Vert ^{2}_{2}
 \end{align}
 where we have used that $R(\theta)\geq d$. On the other hand, Lemma \ref{rip_sampling}  implies that, with probability at least $1-\exp(-pnm)-2\exp(-d/6)$,
 \begin{align}\label{thm_bic_3}
   \Vert \hat \theta_{\Omega}-\theta^{*}_{\Omega}\Vert ^{2}_{2}+4 \;\Q^{2}R(\hat \theta)\geq \frac{p}{2}\Vert \hat \theta-\theta^{*}\Vert ^{2}_{2}.
  \end{align}
  Finally, using Lemma \ref{lemma_bernstein_bernoulli} with $a_{ij}=\left (\theta-\theta^{*}\right )^{2}_{ij}$ and $t=\frac{p}{2}\Vert \theta-\theta^{*}\Vert ^{2}_{2}+4\Q^{2}d$ we get that, with probability at least  $1-\exp(-d)$,
 \begin{align}\label{thm_bic_5}
   \Vert  \theta_{\Omega}-\theta^{*}_{\Omega}\Vert ^{2}_{2}\leq  4\Q^{2}R(\theta)+\dfrac{3p}{2}\Vert \theta-\theta^{*}\Vert ^{2}_{2}
  \end{align}
  where we have used that $R(\theta)\geq d$.
  Plugging \eqref{thm_bic_2} - \eqref{thm_bic_5} in \eqref{thm_bic_1} we get
 \begin{align*} \label{thm_bic_6}
 \dfrac{p}{4} \Vert \hat \theta-\theta^{*}\Vert ^{2}_{2}&\leq  \dfrac{5p}{2}\Vert \theta-\theta^{*}\Vert ^{2}_{2} +8\left (\sigma\vee\Q\right )^{2}R(\hat\theta) +6\left (\sigma\vee\Q\right )^{2}R(\theta)\\&\hskip 0.5 cm+\lambda R(\theta)-\lambda R(\hat\theta)\nonumber
    \end{align*}
with probability larger then $1-5\exp(-d/10)-2\exp\left (-pnm\right )$ where we have used that $d\geq 10
$. Taking here $\lambda=8\left (\sigma\vee\Q\right )^{2}$ finishes the proof. 
% % % % % % % % % % % % % % % % % % % % % % % % % % % % % % % % % % % % % % % % % % % % % % % % % % % % % % % % % % % % % % % % % % % % % %
  \section{Proofs of the lemmas}\label{sec:technical}
  %%%%%%
   \subsection{Lemmas for Theorem \ref{thm:upperfinite}}\label{lemmas_theorem_upperfinite}
   \begin{proof}[Proof of Lemma \ref{lemma_brodten}]
We start by proving (i).   	 Note that for any fixed $X\in\mathcal{A}_{n}$, $\theta=XBZ^T$ belongs to a linear space of dimension not greater than $nm\wedge k_nm = r_n m$ as $BZ^T$ belongs to a linear space of dimension not greater than $k_nm$.
%we can write 
%   	\begin{equation}\label{basis_exp_1}
%   	\theta=XBZ^T=X\left (\sum_{i=1}^{k_n}\sum_{j=1}^{m}a_{ij}(BZ^T)E_{ij}\right )=\sum_{i=1}^{k_n}\sum_{j=1}^{m}a_{ij}(BZ^T)(XE_{ij})
%   	\end{equation}
%   	or 
%   	\begin{equation}\label{basis_exp_2}
%   	\theta=XBZ^T=\sum_{i=1}^{n}\sum_{j=1}^{m}a_{ij}(XBZ^T)E_{ij}.
%   	\end{equation}
%   	Here $E_{ij}$ are canonical basis vectors in the space of matrices of size $k_n\times m$ in \eqref{basis_exp_1} and of size $n\times m$ in \eqref{basis_exp_2}.  
Thus, $\theta-\theta^{*}$ belongs to a linear space of dimension not greater than $r_nm+1$, which we denote by $W_{r_nm}(X)$. We have 
   	\begin{equation*}
   	\underset{\theta\in \Theta(s_n,s_m),\theta\not=\theta^{*}}{\sup} \dfrac{\langle  \theta-\theta^{*},W\rangle^{2}}{\Vert  \theta-\theta^{*}\Vert^{2}_{2}}\leq \underset{X\in\mathcal{A}_n}{\max}\; U_{X}
   	\end{equation*}
   	where, for a fixed $X\in \mathcal{A}_n$, we define
   	\[ U_{X}= \underset{\theta\in \Theta(s_n,s_m),\theta\not=\theta^{*},\theta=XBZ^T}{\sup} \dfrac{\langle  \theta-\theta^{*},W\rangle^{2}}{\Vert  \theta-\theta^{*}\Vert^{2}_{2}}\leq \underset{u\in W_{r_nm}(X):\Vert u\Vert_{2}=1}{\sup} \langle  u,W\rangle.\]
   	It follows from Lemma \ref{lem:kakade}  that
   	\begin{equation*}
   	\mathbb{P}\left \{U_{X}\geq\sigma^{2}\left (2(r_nm+1)+3v\right )\right \}\leq e^{-v},\quad \forall v>0.
   	\end{equation*}
   	Note that, for any $A \in \mathcal{A}_{\dimdone}$, there are at most $\dimsone$ non-zero entries in each row of $A$. This implies that the number of different supports of matrix $A$ is at most ${\dimkone \choose \dimsone}^\dimdone$. For those $\dimdone\dimsone$ non-zero entries, there are at most $|\mathcal{D}_\dimdone|^{\dimdone\dimsone}$ choices. Then, we have
   	\begin{equation}\label{bound_card}
   	\log |\mathcal{A}_\dimdone| \le \dimdone \log {\dimkone \choose \dimsone} + {\dimdone\dimsone} \log |\mathcal{D}_\dimdone| \le \dimdone\dimsone \log \left( \frac{e\dimkone|\mathcal{D}_\dimdone|}{\dimsone} \right).
   	\end{equation}
   	Applying the union bound and using \eqref{bound_card} we get
   	\begin{equation*}
   	\mathbb{P}\left \{\underset{X\in \mathcal{A}_n}{\max}\;U_{X}\geq\sigma^{2}\left (2(r_nm+1)+3v\right )\right \}\leq \left (\dfrac{ek_n|\mathcal{D}_\dimdone|}{s_n}\right )^{s_nn}e^{-v},\quad \forall v>0.
   	\end{equation*}
   	which yields the first result of Lemma \ref{lemma_brodten}. To get the bound on the expectation, we use the fact that for any non-negative random variable $\xi$ and any $a>0$
   	\[ \mathbb{P}(\xi\geq a+t)\geq e^{-t},\quad \forall t>0 \] implies 
   	$ \mathbb{E}\xi\leq a+1. $
   	The proof of (ii) follows the same lines fixing both $X$ and $Z$. To prove (iii) we use that $\theta=XBZ^T$ belongs to a linear space of dimension not greater than $nm$.
   \end{proof}

   \begin{lemma}\label{lem:kakade}
   	Let $\xi$ be a $\sigma$-subgaussian random vector in $\R^n$, and let $\mathcal W$ be a linear subspace of~$\R^n$ with ${\rm dim}(\mathcal W)=d$.  Consider the Euclidean ball $B(0,1)=\{u\in \mathcal W: \, \|u\|_2\le 1\}$. Then, for any $t>0$,
   	\begin{equation*}\label{1kakade}
   	\mathbb P\big(\max_{u\in B(0,1)} \,(u^T\xi)^2 \ge \sigma^2(d+2\sqrt{dt\,} +2t)\big) \le e^{-t}. 
   	\end{equation*}
   \end{lemma}
   \begin{proof}
   	We have 
   $$
   \max_{u\in B(0,1)} \,(u^T\xi)^2 = \max_{u\in B(0,1)} \,(u^T P_{\mathcal W}\xi)^2 = \|P_{\mathcal W}\xi\|_2^2
   $$
   where $P_{\mathcal W}$ is the orthogonal projector onto $\mathcal W$. Applying the following lemma with $A=P_W$ yields the result.
   \begin{lemma}[Hsu et al. \textrm{\cite{kakade}}]\label{lem:kakade0} 
   	Let $\xi$ be a $\sigma$-subgaussian random vector in $\R^n$, and let $A\in \R^{n\times n}$ be a matrix. Set $\Sigma= A^TA$.  Then,  for any $t>0$,
   	\begin{equation*}\label{1kakade0}
   	\mathbb P\big( \|A\xi\|_2^2 \ge   \sigma^2({\rm Tr}(\Sigma)+2\sqrt{{\rm Tr}(\Sigma^2)t} +2\lambda_{\max}(\Sigma)t)\big) \le e^{-t} 
   	\end{equation*}
   	where ${\rm Tr}(\Sigma)$ and $\lambda_{\max}(\Sigma)$ denote the trace and the maximal eigenvalue of $\Sigma$.
   \end{lemma}

   \end{proof}

   \subsection{Lemmas for Theorem \ref{thm:lowerbound}}
   \begin{lemma}\label{lm:KLchi} Assume that the noise variables $W_{ij}$ in model (\ref{eq:modelmc})  are  i.i.d. Gaussian with distribution
    $\mathcal{N}(0, \sigma^2)$.
  Then,  the Kullback-Leibler divergence between $\mathbb{P}_{\theta}$ and  $\mathbb{P}_{\theta'}$ has the form 
  $$KL\left(\mathbb{P}_{\theta}, \mathbb{P}_{\theta'}\right)=\frac{p}{2\sigma^2} \| \theta-\theta'\|_2^2.$$
   \end{lemma}
   Proof of this lemma is straighforward and it is therefore omitted.
   \begin{lemma}\label{lm:select} 
   Let $k\ge 2$ and $s\ge 1$ be integers, $s\le k$. There exists a subset $S_0$ of the set of binary sequences $\{0,1\}^{k}$ such that 
   \begin{itemize}
   \item[(i)] $\log |S_0|\ge c_1^* s \log \frac{ek}{s}$,
   \item[(ii)] $c_2^*s \le \|a\|_0 \le s$ for all $a \in S_0$, and $\|a\|_0 = s$ for all $a \in S_0$ if $s\le k/2$,
   \item[(iii)] $\|a-b\|_2^2 \ge  c_3^* s$ for all $a, b \in S_0$ such that $a\ne b$,
   \end{itemize}
where $c_j^*>0$, $j=1,2,3$, are absolute constants.
   \end{lemma}
    \begin{proof} For $s\le k/2$ the result follows from Lemma A.3 in \cite{RigTsy2011}. For $k/2<s\le k$ and $k\ge 32$, we restrict the consideration only to binary sequences in $\{0,1\}^{k}$ such that the first $m=\lceil k/4 \rceil $ elements can be either 0 or 1, the last $s-m$  elements are 1 and the remaining elements are 0. Then, (i) - (iii) follow from the Varshamov-Gilbert bound \cite[Lemma 2.9]{tsy_09} applied to the set of binary sequences of length $m$. For $k/2<s\le k$ and $k < 32$, the result is obvious.
     \end{proof}
%  \begin{proof} %[Proof of Lemma \ref{lm:KLchi}]
%  From (\ref{eq:modelmc}), the density of $Y_{ij}$ can be written as a mixture model 
%  $$ \mathbb{P}_{\theta_{ij}} (y_{ij}) = (1-p) \delta_0 + \frac{p}{\sqrt{2\pi \sigma^2}} \exp \left( - \frac{(y_{ij} - \theta_{ij})^2}{2\sigma^2} \right),$$
%  where $\delta_0$ is the point mass on $y_{ij}=0$. Since the entries of $Y$ are independent, we have
%  \begin{eqnarray*}
%  KL\left(\mathbb{P}_{\theta}, \mathbb{P}_{\theta'}\right) &=& \sum_{i,j} KL\left(\mathbb{P}_{\theta_{ij}}, \mathbb{P}_{\theta'_{ij}}\right) \\
%  &=& \sum_{i,j} \left((1 - p) \log \frac{1-p}{1-p} + p KL \left(\mathcal{N}(\theta_{ij}, \sigma^2), \mathcal{N}(\theta_{ij}', \sigma^2)\right) \right) \\
%  &=& \frac{p}{2\sigma^2} \| \theta - \theta' \|_2^2.
%  \end{eqnarray*}
%  %Olga:OK
%  \end{proof}
  
  % % % % % % % % % % % % % % % % % % % % % % % % % % % % % % % % % % % % % % % % % % % % %% % % % %
  % \subsection{Lemma \ref{lm:specialB}}
  \begin{lemma} \label{lm:specialB}
   Let $\{a_1, a_2, \dots, a_N\} \subseteq \{0,1\}^{k}$.
   % with $\|a_i\|_0 = s$ for all $i \in [N]$. 
    Let $r$ be an integer satisfying $r> 96 \log N$. Then, there exists a matrix $Q \in \{-1,1\}^{r\times k}$ such that for any $u, v \in [N]$,
   \begin{equation*}
   \frac{r}{2} \|a_u - a_v\|_2^2 \le \| Qa_u - Qa_v\|_2^2 \le \frac{3r}{2} \|a_u- a_v\|_2^2.
   \end{equation*}
   \end{lemma}
   \begin{proof} The result follows immediately from Johnson - Lindenstrauss Lemma as stated in  \cite[Theorem 2]{achlioptas} by taking there $\beta=1$ and $\epsilon=1/2$. 
   \end{proof}

  \subsection{Lemmas for Theorem \ref{thm:oracleupper}}
  
   % % % % % % % % % % % % % % % % % % % % % % % % % % % % % % % % % % % % % % % % % % % % % % % % % % % % % % % % % % % % % % 
  \begin{lemma} \label{lm:chaining}
  Let $R=\sqrt{\frac{3(\Q^2+\sigma^2)}{p} N \epsilon_0^2}$ and $\Theta^R = \left\{ \theta \in \Theta,  \| \theta - \theta_0\|_2 \ge 2^{s^{*}}R  \right\}$. Then we have
  \begin{equation*}
  \mathbb{P} \left\{ \sup_{\theta \in \Theta^R} \left( \left \langle \theta-\theta_0, W \right \rangle - \frac{1}{8} \|\theta-\theta_0 \|_2^2 \right) > 96 R^2 \right\} \le 4 \exp\left ( - \frac{\alpha R^2}{2}\right ) +  \exp(-pN/6)
  \end{equation*}
  where $\alpha = \frac{p}{6(\Q^{2}+\sigma^{2})}$.
  \end{lemma}
  \begin{proof} [Proof of Lemma \ref{lm:chaining}]
  Let $\mathcal{E}=\{\|W\|_2 \le \sqrt{\frac{N}{2\alpha}}\}$ and for $s\geq 2$ let $ \Theta^R_{s} =\{ \theta \in \Theta^R,  2^{s} R \le \| \theta - \theta_0\|_2 \le 2^{s+1} R \} $. Then we have that
  $\Theta^{R} = \cup_{s=s^{*}}^{\infty} \Theta^R_s$. The union bound yields
  \begin{eqnarray*}
  && \mathbb{P} \left\{ \sup_{\theta \in \Theta^R} \left( \langle \theta-\theta_0, W \rangle - \frac{1}{8} \|\theta-\theta_0 \|_2^2 \right) > 96 R^2 , \mathcal{E} \right\} \\
  &\le& \sum_{s=s^{*}}^{\infty} \mathbb{P} \left\{ \sup_{\theta \in \Theta^R_s} \left( \langle \theta-\theta_0, W \rangle - \frac{1}{8}\|\theta-\theta_0 \|_2^2 \right) > 96R^2, \mathcal{E} \right\} \\
  &\le& \sum_{s=s^{*}}^{\infty} \mathbb{P} \left\{ \sup_{\theta \in \Theta^R_s} \langle \theta-\theta_0, W \rangle \ge \left( 2^{2s-3} + 96 \right)  R^2, \mathcal{E} \right\} \\
  &\le& \sum_{s=s^{*}}^{\infty} \mathbb{P} \left\{ \sup_{\theta \in \Theta_{2^{s+1}R}} \langle \theta-\theta_0, W \rangle \ge  \left( 2^{2s-3} + 96 \right) R^2, \mathcal{E} \right\}.
  \end{eqnarray*}
  where the last step is due to the fact that $\Theta_{s}^R \subseteq \Theta_{2^{s+1}R}$.
  Now we are going to apply Lemma \ref{lm:empirical} with $D=2^{s+1}R$, $\epsilon = D \epsilon_0$ and $t=(2^{2s-4}+48)R^2$. It is easy to check that $t \in [DR, D^2]$ for all $s \ge s^{*}$. Then (\ref{eq:empirical}) yields 
  $$  \sum_{s=s^{*}}^{\infty} \mathbb{P} \left\{ \sup_{\theta \in \Theta_{2^{s+1}R}} \langle \theta-\theta_0, W \rangle \ge  \left( 2^{2s-3} + 96 \right) R^2, \mathcal{E} \right\} \le \sum_{s=2}^{\infty} 2 e^{-\frac{\alpha s R^2}{4}} \le 4 e^{-\frac{\alpha R^2}{2}}.$$
  The last inequality holds when $\alpha R^2 \ge 4$. By Lemma \ref{lm:bernstein}, $P(\mathcal{E}^c) \le  \exp(-pN/6)$, and therefore we obtain the desired result. 
  \end{proof}

  % % % % % % % % % % % % % % % % % % % % % % % % % % % % % % % % % % % % % % % % % % % % % % % % % % % % % % % % % % % % % % % % % % % % % % % % % %

  \begin{lemma}\label{lm:empirical}
  Suppose $\epsilon$ satisfies $\sqrt{N} \epsilon \le 2D \sqrt{\log \mathcal{N}_\epsilon(\Theta_D)}$. Then, for $D \ge \sqrt{2\log \mathcal{N}_\epsilon(\Theta_D)/\alpha}$ and for any $t \in [D\sqrt{2\log \mathcal{N}_\epsilon(\Theta_D)/\alpha}, D^2]$, we have
  \begin{equation} \label{eq:empirical}
  \mathbb{P} \left \{ \sup_{\theta \in \Theta_D} \langle \theta - \theta_0, W \rangle \ge 2t, \mathcal{E} \right\} \le 2 \exp\left( - \frac{\alpha t^2}{2D^2} \right)
  \end{equation}
  where $\mathcal{E}=\{\|W\|_2 \le \sqrt{\frac{N}{2\alpha}}\}$. 
  \end{lemma}
  \begin{proof}[Proof of Lemma \ref{lm:empirical}]
  Let $\mathcal{C}_D$ be an $\epsilon$-covering set of $\Theta_D$ under the Frobenius norm. That is, for any $\theta \in \Theta_D$, there exists $\overline{\theta} \in \mathcal{C}_D$ such that $\| \theta - \overline{\theta}\|_2 \le \epsilon$. Denote by $N \triangleq \mathcal{N}_\epsilon(\Theta_D)$ the minimum cardinality of such $\mathcal{C}_D$. This yields
  \begin{eqnarray*} 
  \langle \theta - \theta_0, W \rangle &=& \langle \overline{\theta} - \theta_0, W \rangle + \langle \theta - \overline{\theta}, W \rangle \\
  &\le& \max_{\overline{\theta} \in \mathcal{C}_D}  \langle \overline{\theta} - \theta_0, W \rangle +  \epsilon \| W \|_2, 
  \end{eqnarray*}
  where we  use Cauchy-Schwarz for the last inequality.  On the event $\mathcal{E}$, we have that  $\epsilon \| W \|_2 \le \epsilon \sqrt{\frac{N}{2\alpha}} \le D\sqrt{2\log (\mathcal{N}_\epsilon(\Theta_D))/\alpha}$. It implies
  $$  \mathbb{P} \left \{ \sup_{\theta \in \Theta_D} \langle \theta - \theta_0, W \rangle \ge 2t ,~\mathcal{E} \right\} \le \mathbb{P} \left \{ \max_{\overline{\theta} \in \mathcal{C}_D}  \langle \overline{\theta} - \theta_0, W \rangle \ge t  \right\}  $$
  for $t \ge D\sqrt{2\log \mathcal{N}_\epsilon(\Theta_D)/\alpha}$. By union bound and Lemma \ref{lm:bernstein}, the right hand side of the above inequality can be bounded from above by 
  \begin{eqnarray*}
  \sum_{\overline{\theta} \in \mathcal{C}_D} \mathbb{P} \left \{ \langle \overline{\theta} - \theta_0, W \rangle \ge t  \right\} \le \exp \left ( -\alpha \min\{t^2/D^2, t)\} + \log \mathcal{N}_\epsilon(\Theta_D)\right).
  \end{eqnarray*}
  Then the desired result (\ref{eq:empirical}) holds when $ D\sqrt{2\log N/\alpha} \le t \le D^2$.
 %Olga:ok
  \end{proof}
  
  % % % % % % % % % % % % % % % % % % % % % % % % % % % % % % % % % % % % % % % % % % % % % % % % % % % % % % % % % % % % % % % % % % % % % % % % % % % % % % % % % % % % % % % % % % % % % % %

 \begin{lemma} \label{lm:bernstein}
   Let $\alpha= \frac{p}{6(\Q^2+\sigma^2)}$. Then, for any $t>0$ and $a\in \mathbb{R}^{n}$, we have
   \begin{equation} \label{eq:bernestein1}
   \mathbb{P} \left\{ \langle a, W \rangle > t \right\} \le \exp \left( - \alpha \min\left\{\frac{t^2}{\|a\|^2}, \frac{\Q t}{\|a\|_\infty} \right\} \right),
   \end{equation}
   and
   \begin{equation}\label{eq:bernestein2}
   \mathbb{P} \left \{ \|W\|_2 \ge \sqrt{\frac{N}{2\alpha}} \right\} \le \exp ( - pN/6).
   \end{equation}
   \end{lemma}
  \begin{proof}[Proof of Lemma \ref{lm:bernstein}]
    Since $W_{i}  = \theta_{i}^* \frac{E_{i} - p}{p} + \xi_{i}\frac{E_{i}}{p}$,
    for $\lambda (\sigma \vee 2\Q) \Vert a\Vert_{\infty}\le p$, we have
    $$ \bE \left (e^{\lambda a_i W_i}\right ) \le \bE \left (e^{\lambda a_i\theta^{*}_{i} \frac{E_{i} - p}{p}} e^{\frac{\lambda^2 a_i^{2}\sigma^2 E_i}{2p^2}}\right ) \le e^{\frac{\lambda^2 a_i^{2}\sigma^2}{2p}} e^{\left(\frac{\lambda a_i\theta_i^*}{p}+\frac{\lambda^2 a_i^{2} \sigma^2}{2p^2}\right)^2 p} \le  e^{\frac{3\lambda^2 a_i^{2}(\Q^2 + \sigma^2)}{2p}}.$$
    Here the second inequality is due to the fact that $\bE e^{\lambda(E_i - p)} \le e^{\lambda p}$ for $|\lambda| \le 1$.  Now, following the Chernoff argument as in the proof of Lemma \ref{lemma_bernstein_1} we get \eqref{eq:bernestein1}.
    
    To prove \eqref{eq:bernestein2} note first that the variance of $W_i$ satisfies 
    $$  \mathbb{E} W_{i} ^2 = \theta_{i}^2 \frac{1-p}{p} + \frac{\mathbb{E}\xi_{i}^2}{p} \le \frac{\Q^2+\sigma^2}{p}. $$
    Then we have
    $$ W_{i} ^2 - \mathbb{E} W_{i} ^2 = (\theta_{i}^*)^2 \frac{(E_{i} - p)(1-2p)}{p^2} +  \frac{E_{i}}{p^2} (\xi_{i}^2 - \mathbb{E} \xi_{i}^2 ) + \frac{E_{i} - p}{p^2} \mathbb{E} \xi_{i}^2 + \frac{E_{i}(E_{i} - p)}{p^2} 2 \theta_{i}^* \xi_{i} $$
    When $\lambda (\Q^2 + \sqrt{2}\sigma^2) / p^2 \le 1$, we obtain
    \begin{eqnarray*}
    \mathbb{E} \left (e^{\lambda(W_{i} ^2 - \mathbb{E} W_{i} ^2)}\right ) &\le& 
     \mathbb{E} \left (e^{\lambda(\theta_{i}^*)^2 \frac{(E_{i} - p)(1-2p)}{p^2}} e^{\frac{2\lambda^2\sigma^4 E_{i}}{p^4}} e^{\frac{\lambda\sigma^{2}(E_{i} - p)}{p^2}}e^{\frac{2\lambda^2 \Q^2 \sigma^2}{p^4}E_{ij}(E_{ij}-p)^2} \right )\\&\le&
    %e^{\frac{\lambda^2 (\Q^4+2\sigma^4)}{p^3}} e^{\frac{2\lambda^2\sigma^4 E_{ij}}{p^4}} %e^{\frac{4\lambda^2 M^2 \sigma^2}{p^4}E_{ij}(E_{ij}-p)^2} 
    e^{\frac{6\lambda^2 (\Q^4+\sigma^4)}{p^3}}. 
    \end{eqnarray*}
    The Chernoff argument yields
    $$ \mathbb{P} \left\{ \sum_{i=1}^{N} \left( W_{i} ^2 - \mathbb{E} W_{i} ^2 \right) \ge t  \right\} \le \exp\left\{ - \lambda t + \frac{6\lambda^2 (\Q^4+\sigma^4)}{p^3} N \right\}.$$
    For $t = \frac{2(\Q^2+\sigma^2)}{p}N $, we choose $\lambda= \frac{p^2}{6(\Q^2+\sigma^2)}$ to get the desired result. 
    %Olga:ok
    \end{proof} 
  
% % % % % % % % % % % % % % % % % % % % % % % % % % % % % % % % % % % % % % % % % %  
  \subsection{Lemmas for Proposition \ref{prop:coveringbound} }
  \begin{lemma} \label{lm:coveringB}
    For any fixed $X \in \mathbb{R}^{n \times k_n}$ and $Z\in \mathbb{R}^{m \times k_m}$, let $$\mathcal{T}_R(X,Z)=\left \{\mtheta=XBZ^T, B \in \mathbb{R}^{k_n \times k_m}, ||\theta-\theta_0||_2 \le R \right \}.$$ Then, for any $0< \epsilon \le R$
    \begin{equation*}
    \mathcal{N}_{\epsilon}\left (\mathcal{T}_R(X,Z)\right )\le \left( \frac{3R}{\epsilon} \right)^{r_n r_m}.
    \end{equation*}
    \end{lemma}
   
  \begin{proof} 
  	Note that for any fixed $X \in \mathbb{R}^{n \times k_n}$ and $Z\in \mathbb{R}^{m \times k_m}$ the set of matrices $\left \{\mtheta=XBZ^T, B \in \mathbb{R}^{k_n \times k_m}\right \}$ belongs to a linear subspace of $\mathbb{R}^{n \times m}$ of dimension at most $r_nr_m$. To see it, note that any of such matrices $\mtheta$ can be written as $\mtheta=\sum_{i=1}^{r_n}\sum_{j=1}^{r_m} a_{ij}(\mtheta) M_{ij}$ with scalars $a_{ij}(\mtheta)\in\mathbb R$ and matrices $M_{ij}\in\mathbb R^{n\times m}$. Now, applying the standard bound on the covering number of the ball in the Euclidean norm (see, e.g. Lemma 5.2 in \cite{vershynin}) we get the result of the lemma.

  \end{proof}
  % % % % % % % % % % % % % % % % % % % % % % % % % % % % % % % % % % % % % % % % % % % % % % % % % % % % % % % % % % % % % % % % % % %
   \begin{lemma} \label{lm:coveringZ}
      We have the following upper bound on the  $\epsilon$-covering number of $\mathcal{A}_n$ under $\ell_{\infty}$ norm: 
       \begin{equation*}
       \mathcal{N}_\epsilon\left ( \mathcal{A}_{n} ,\|\cdot \|_{\infty}\right ) \le {\dimk_n \choose \dims_n}^n \left( \frac{1}{\epsilon}  \right)^{n\dims_n}.
       \end{equation*}
       \end{lemma}
  
  \begin{proof} %[Proof of Lemma \ref{lm:coveringZ}]
  Note that there are ${k_n \choose s_n } ^ n$ subsets of $\{1,\dots, k_n\}^n$ that satisfy the column sparsity constraint of $A \in \mathcal{A}_n$. For any such subset, the selected $n s_n$ entries lie in the unit Euclidean ball $\mathbb{B}_2(1)$. By the standard volume ratio argument we can find an $\epsilon-$covering set of $\mathbb{B}_2(1)$ with at most $(\frac{1}{\epsilon})^{n s_n}$ elements. Hence, the lemma follows. 
  \end{proof}
% % % % % % % % % % % % % % % % % % % % % % % % % % % % % % % % % % % % % % % % % % % % % % % % % % % % % % % % % % % % % %  
 \begin{lemma} \label{lm:matrixbound}
  Assume that $A \in \mathbb{R}^{n \times k}$, $Z \in \mathbb{R}^{m \times k}$ and  that each row of $Z$ is $s-$sparse. Then, 
  $$ \|AZ^T\|_2 \le s \sqrt{mn} \|A\|_\infty \|Z\|_\infty. $$
  \end{lemma}

  \begin{proof} % [Proof of Lemma \ref{lm:matrixbound}]
We have
  \begin{eqnarray*}
  \|AZ^T\|_2^2 = \sum_{i \in [n]} \sum_{j \in [m]} \left( \sum_{l \in [k]} A_{il} Z_{jl}  \right)^2  &\le &\sum_{i \in [n]} \sum_{j \in [m]} \left( s \|A\|_\infty \|Z\|_\infty \right)^2 \\&\le& s^2 nm \|A\|_\infty^2 \|Z\|_\infty^2.
  \end{eqnarray*}
  \end{proof}
  % % % % % % % % % % % % % % % % % % % % % % % % % % % % % % % % % % % % % % % % % % % % % % % % % % % % % % % % % % % % % %
  \subsection{Lemmas for Theorem \ref{thm_mc_bic}}
   
%      \begin{lemma}\label{union_bernstein_1}
%      Assume that $\epsilon \le \dfrac{t}{\sigma\sqrt{7N}}$. Then for any $t>0$  we have
%      \begin{equation} \label{eq:empirical}
%      \mathbb{P} \left \{ \sup_{\theta \in \Theta_D(s_n,s_m)} \langle \theta - \theta_0, W_{\Omega} \rangle \ge 2t, \mathcal{E} \right\} \le  \exp\left( - \min\left \{\frac{ t^{2}}{4\sigma^{2}pD^{2}},\frac{ t}{\sqrt{2}\sigma\Q}\right \} +\log(N(\epsilon,D)\right)
%      \end{equation}
%       Here $\mathcal{E}=\{\sqrt{\sum_{ij} (E_{ij} W^{2}_{ij})} \ge \sigma\sqrt{7N}\}$. 
%      \end{lemma}
%      
%     \begin{proof}
%     The proof of this lemma follows the lines of the proof of Lemma  \ref{lm:empirical} using Lemma \ref{lemma_bernstein_1} and the union bound. 
%     \end{proof} 
%  

      % % % % % % % % % % % % % % % % % % % % % % % % % % % % % % % % % % % % %         

             % % % % % % % % % % % % % % % % % % % % % % % % % % % % % % % % % % % % % % % % % % % % % % % % % %
  
   \begin{lemma}\label{rip_sampling}
    Assume that $n\,m\log\left (3\sqrt{nm}\right )\geq 6\log \left (k_n\,k_m\right )$. Then, with probability larger then $1-2\exp(-d/6)-\exp(-pnm)$
    \begin{align*}
    \underset{\theta\in\mathcal{X}_{\nu}}{\sup}\dfrac{ \frac{p}{2}\Vert {\theta}-\theta^{*}\Vert^{2}_{2}-\Vert  \theta_{\Omega}-\theta^{*}_{\Omega}\Vert ^{2}_{2}}{\Q^{2} R(\theta)}\leq 4
    \end{align*}
  where $R(\theta)$ is defined in \eqref{def_R}. 
    \end{lemma}
  \begin{proof}
    Let $\mathcal{E}_{2}=\left \{ \sum_{ij} (E_{ij}-p)^{2}\le 3pnm\right \}$. Lemma \ref{lemma_bernstein_bernoulli} implies that $\mathbb{P}(\mathcal{E}_{2})\geq 1-\exp\left (-pnm\right )$. Using the definition of $\mathcal{X}_{\nu}$ we have that
   \begin{align}\label{rip_2}
  & \mathbb{P}\left \{\underset{\theta\in\mathcal{X}_{\nu}}{\sup}\dfrac{ \frac{p}{2}\Vert {\theta}-\theta^{*}\Vert^{2}_{2}-\Vert  \theta_{\Omega}-\theta^{*}_{\Omega}\Vert ^{2}_{2}}{\Q^{2} R(\theta)}\geq 4, \mathcal{E}_{2}\right \}
  \leq \sum_{s_n=1}^{k_n}\sum_{s_m=1}^{k_m} \mathrm I_{s_n,s_m}
  \end{align} 
  where
  $$
   \mathrm I_{s_n,s_m} = {\mathbb{P}\left \{\underset{\theta\in \widetilde\Theta^{\nu}(s_n,s_m)}{\sup}\frac{p}{2}\Vert {\theta}-\theta^{*}\Vert^{2}_{2}-\Vert  \theta_{\Omega}-\theta^{*}_{\Omega}\Vert ^{2}_{2}\geq 4\, \Q^{2}R(s_n,s_m), \mathcal{E}_{2}\right \}}
  $$
   and $\widetilde\Theta^{\nu}(s_n,s_m)\triangleq \{\theta \in \widetilde\Theta(s_n,s_m), \|\theta - \theta^{*}\|_2 > 2^{6}\nu\}$. In order to bound ${\mathrm I}_{s_n,s_m}$ from above, we use a standard peeling argument. Let $\mu=2$. For $l\in\mathbb N$ set $$S_l=\left \{\theta\in \widetilde\Theta(s_n,s_m)\,:\,\mu^{l}\nu \leq \Vert \theta-\theta^{*}\Vert _2\leq \mu^{l+1}\nu\right \}.$$
  Then, $\widetilde\Theta^{\nu}(s_n,s_m)=\underset{l=7}{\overset{\infty}{\cup}}S_{l}$ and the union bound yields
    \begin{align}\label{rip_1}
  \mathrm I_{s_n,s_m}&\leq  \sum_{l=7}^{\infty}\mathbb{P}\left \{\underset{\theta\in S_{l}}{\sup}\frac{p}{2}\Vert {\theta}-\theta^{*}\Vert^{2}_{2}-\Vert  \theta_{\Omega}-\theta^{*}_{\Omega}\Vert ^{2}_{2}\geq 4\, \Q^{2}R(s_n,s_m), \mathcal{E}_{2}\right \}\nonumber\\
  &\leq  \sum_{l=7}^{\infty}\mathbb{P}\left \{\underset{\theta\in \widetilde\Theta_{r}(s_n,s_m)}{\sup}p\Vert {\theta}-\theta^{*}\Vert^{2}_{2}-\Vert  \theta_{\Omega}-\theta^{*}_{\Omega}\Vert ^{2}_{2}\geq 4\, \Q^{2}R(s_n,s_m)+\frac{pr^{2}}{2\mu^{2}}, \mathcal{E}_{2}\right \}
    \end{align}
  where $r= \mu^{l+1}\nu$. 
 We have that $\mathbb{E}\left (E_{ij}\left (\theta-\theta^{*}\right )^{2}_{ij}\right )=p\left (\theta-\theta^{*}\right )^{2}_{ij}$ and
    \begin{align*}
    p\Vert {\theta}-\theta^{*}\Vert^{2}_{2}-\Vert \theta_{\Omega}-\theta^{*}_{\Omega}\Vert ^{2}_{2}=\sum_{(ij)}(p-E_{ij})\left (\theta-\theta^{*}\right )^{2}_{ij}. 
    \end{align*}
    Let $S=\min\left \{s\geq 2\;:\;2^{-s}\leq \sqrt{\frac{\mathcal{I}}{nm}}\right \}$ and 
    $\{\widetilde G^{S}_{j}\}^{N_s}_{j=1}$ be a minimal $2^{-S}r$-covering set of $\Theta_{r}(s_n,s_m)$ in Frobenius norm given by Proposition \ref{prop:coveringbound}. Let $G^{S}_{j}=\Pi(\widetilde G^{S}_{j})$ where $\Pi$ is the projection operator under the Frobenius norm into the set $$\mathcal{B}=\{\theta\in \mathbb R^{n\times m}\;:\; \Vert \theta\Vert_{\infty}\leq \Q\}.$$
  As $\mathcal{B}$ is closed and convex, $\Pi$ is non-expansive and we have that for any $\theta\in \widetilde\Theta_{r}(s_n,s_m)$, $\Vert \theta - G^{S}_{j}\Vert_{2}\leq \Vert \theta - \widetilde G^{S}_{j}\Vert_{2}$.
   Then, there exists $G_{\theta}^{S}\in \{G^{S}_{j}\}^{N_s}_{j=1}$ such that on the event $\mathcal{E}_{2}$
  \begin{align*}
   \left \vert\sum_{(ij)}(p-E_{ij})\left [\left (\theta^{*}-\theta\right )^{2}_{ij}-\left (\theta^{*}-G_{\theta}^{S}\right )^{2}_{ij}\right ]\right \vert\leq \frac{pr^{2}}{4\mu^{2}}
  \end{align*}
  where we use $\left (\theta^{*}-\theta\right )^{2}_{ij}-\left (\theta^{*}-G_{\theta}^{S}\right )^{2}_{ij}=\left ( G_{\theta}^{S}-\theta\right )_{ij}\left ( 2\theta^{*}-G_{\theta}^{S}-\theta\right )_{ij}$, Cauchy-Schwarz inequality and $\frac{pr}{16\Q\mu^{2}\sqrt{3pnm}}\geq \sqrt{\frac{\mathcal{I}}{nm}}$.
  So it suffices to prove the exponential inequality for
  \begin{align*}
  \mathbb{P}\left \{\underset{k=1,\dots,N_{S}}{\max}\sum_{(ij)}(p-E_{ij})(\theta^{*}-G^{S}_{k})^{2}_{ij}\geq 4\, \Q^{2}R(s_n,s_m)+\frac{pr^{2}}{4\mu^{2}}\right \}.
  \end{align*}
  We apply Markov's inequality. Set $t= R(s_n,s_m)+\frac{pr^{2}}{16\Q^{2}\mu^{2}}$
  \begin{align}\label{rip_3}
    &\mathbb{P}\left \{\underset{k=1,\dots,N_{S}}{\max}\dfrac{1}{4\Q^{2}}\sum_{(ij)}(p-E_{ij})(\theta^{*}-G^{S}_{k})^{2}_{ij}\geq t\right \}\nonumber\\&\hskip 0.5 cm\leq e^{-t}\E \exp\left (\underset{k=1,\dots,N_{S}}{\max}\dfrac{1}{4\Q^{2}}\sum_{(ij)}(p-E_{ij})(\theta^{*}-G^{S}_{k})^{2}_{ij}\right )\nonumber\\&\hskip 1 cm \leq e^{-t}\sum_{k=1}^{N_{S}}\E \exp\left (\dfrac{1}{4\Q^{2}}\sum_{(ij)}(p-E_{ij})(\theta^{*}-G^{S}_{k})^{2}_{ij}\right )
   \nonumber \\&\hskip 1.5 cm \leq e^{-t}\sum_{k=1}^{N_{S}}\prod_{(ij)}\E \exp\left ((p-E_{ij})\dfrac{(\theta^{*}-G^{S}_{k})^{2}_{ij}}{4\Q^{2}}\right ).
    \end{align}
  Now we use the following lemma that follows easily from \cite[p.22]{Hoeffding}, see also \cite{gao2015optimal}:
  \begin{lemma}\label{lm:exp_bernoulli}
  Let $E\sim \mathrm{Ber}(p)$, then for any $\vert \lambda\vert\leq 1$ we have
  \begin{equation*}
  \E \exp\{\lambda(p-E)\}\leq e^{\lambda^{2}p}.
  \end{equation*}
  \end{lemma}
  Lemma \ref{lm:exp_bernoulli} and \eqref{rip_3} imply
    \begin{align*}
        &\mathbb{P}\left \{\underset{k=1,\dots,N_{S}}{\max}\dfrac{1}{4\Q^{2}}\sum_{(ij)}(p-E_{ij})(\theta^{*}-G^{S}_{k})^{2}_{ij}\geq t\right \}
        \\&\hskip 2 cm \leq e^{-t}\sum_{k=1}^{N_{S}}\prod_{(ij)} \exp\left (\dfrac{p(\theta^{*}-G^{S}_{k})^{4}_{ij}}{16\Q^{4}}\right )\\&\hskip 2 cm\leq   \exp\left (\dfrac{pr^{2}}{2^{8}\Q^{2}}+\log(N_{S})-t\right )\leq \exp\left (-\dfrac{3pr^{2}}{64\mu^{2}\Q^{2}}\right )
        \end{align*}
    where we use $S\geq 3$ and Lemma \ref{lemma_control_covering} which implies $\log(N_{S})\leq R(s_n,s_m)$. Putting this last bound into \eqref{rip_1} and \eqref{rip_2} we get
    \begin{align*}
     &\mathbb{P}\left \{\underset{\theta\in\mathcal{X}_{\nu}}{\sup}\dfrac{ \frac{p}{2}\Vert {\theta}-\theta^{*}\Vert^{2}_{2}-\Vert  \theta_{\Omega}-\theta^{*}_{\Omega}\Vert ^{2}_{2}}{\Q^{2} R(\theta)}\geq 4, \mathcal{E}_{2}\right \}\\&\hskip 0.5 cm\leq \sum_{s_n=1}^{k_n}\sum_{s_m=1}^{k_m}\sum_{l=6}^{\infty} \exp\left (-\dfrac{3l\log(\mu)\mathcal{I}}{32}\right )\\&\hskip 1 cm\leq 2\sum_{s_n=1}^{k_n}\sum_{s_m=1}^{k_m}\exp\left (-\dfrac{\mathcal{I}}{3}\right )\leq 2\exp\left (-\dfrac{\mathcal{I}}{6}\right )
    \end{align*}
    where we used $e^{x}\geq x$ and $\nu^{2}=\frac{\left (\sigma\vee\Q\right )^{2}}{p}\mathcal{I}$. In the last inequality we use condition $n\,m\log (3\sqrt{n\wedge m})\geq 6\log \left (k_n\,k_m\right )$ which implies $\frac{\mathcal{I}}{6}\geq \log \left (k_n\,k_m\right )$. Finally, $\mathcal{I}=\underset{(s_n,s_m)}{\min} R(s_n,s_m)\geq (n+m)$ implies the result of Lemma \ref{rip_sampling}.
  \end{proof} 
   
 % % % % % % % % % % % % % % % % % % %
 
  % % % % % % % % % % % % % % % % % % % % % % % % % % % % % % % % % % % % % % % % % %
       \begin{lemma}\label{lemma_bernstein_bernoulli}
         Let $a=(a_{ij})\in\mathbb{R}^{nm}$. Then, for any $t>0$ we have
         \begin{align}\label{eq_bernstein_bernoulli_1}
         \mathbb{P} \left\{ \sum_{ij} (E_{ij}-p)a_{ij} > t \right\} \le \exp \left( -  \min\left\{\frac{t^2}{4p\|a\|_2^2}, \frac{t}{2\|a\|_\infty} \right\} \right),
         \end{align}
         and 
         \begin{equation}\label{eq_bernstein_bernoulli_2}
         \mathbb{P} \left \{ \sum_{ij} (E_{ij}-p)^{2} \ge 3pnm \right\} \le \exp ( -pnm).
         \end{equation}
         %where $N=pnm$.
         \end{lemma}
        \begin{proof}
     This is a direct consequence of  Bernstein's inequality. For \eqref{eq_bernstein_bernoulli_1} we note that \begin{align*}
          \E \sum_{ij}(E_{ij}-p)^{2}a_{ij}^{2}\leq p(1-p)\|a\|_2^2.
          \end{align*}
 On the other hand, for any $(ij)\in[n]\times[m]$ we have $\left \vert (E_{ij}-p)a_{ij}\right \vert\leq \|a\|_\infty$. Then, Bernstein's inequality implies \eqref{eq_bernstein_bernoulli_1}.

     For \eqref{eq_bernstein_bernoulli_2} we have 
     \begin{align*}
     \E \sum_{ij}(E_{ij}-p)^{4}\leq \E \sum_{ij}(E_{ij}-p)^{2}= p(1-p)nm
     \end{align*}
     and 
     \begin{align*}
          \left \vert E_{ij}-p\right \vert^{2} \leq p^{2}\vee (1-p)^{2}\leq 1.
          \end{align*}
          Then, using Bernstein's inequality we get 
          \begin{align*}
          \mathbb{P}\left \{\sum_{ij}(E_{ij}-p)^{2}\geq t+pnm(1-p)\right \}\leq \exp\left \{-\frac{t^{2}}{2pnm(1-p)+2t/3}\right \};
          \end{align*}
     Choosing $t=2pnm$ we get the result of Lemma     \ref{lemma_bernstein_bernoulli}. 
        \end{proof}
        % % % % % % % % % % % % % % % % % % % % % % % % % % % % % % % % % % % % %
          % % % % % % % % % % % % % % % % % % % % % % % % % % % % % % % % % % % % % % % % % % % % % % % %
           
           \begin{lemma}\label{stoch}
           Assume that  $n\,m\log\left (3\sqrt{nm}\right )\geq 6\log \left (k_n\,k_m\right )$. Then, with probability greater then $1-2\exp(-d/10)-\exp\left (-pnm\right )$,
           \begin{align*}
           \underset{\theta\in\mathcal{X}_{\nu}}{\sup}\dfrac{\langle \xi_{\Omega},\theta-\theta^{*}\rangle-\frac{p}{8} \|\theta-\theta^{*} \|_2^2}{\left (\sigma\vee\Q\right )^{2}R( \theta)}\leq 2.
           \end{align*}
           where $R(\theta)$ is defined in \eqref{def_R}, and $\mathcal{X}_{\nu}$ is defined in the proof of Theorem \ref{thm_mc_bic}.
            \end{lemma}         
           \begin{proof} 
          We proceed as in the proof of Lemma \ref{rip_sampling}.  Let $\mathcal{E}_{1}=\left \{ \sqrt{\sum_{ij} E_{ij} \xi^{2}_{ij}} \le \sigma\sqrt{6pnm}\right \}$. Lemma \ref{lemma_bernstein_1} implies that $\mathbb{P}(\mathcal{E}_{1})\geq 1-\exp\left (-pnm\right )$. Using the definition of $\mathcal{X}_{\nu}$ we get
           \begin{align}\label{stoch_1}
          & \mathbb{P}\left \{\underset{\theta\in\mathcal{X}_{\nu}}{\sup}\dfrac{\langle \xi_{\Omega},\theta-\theta^{*}\rangle-\frac{p}{8} \|\theta-\theta^{*} \|_2^2}{\left (\sigma\vee\Q\right )^{2}R( \theta)}\geq 2, \mathcal{E}_{1}\right \}\leq  \sum_{s_n=1}^{k_n}\sum_{s_m=1}^{k_m}{\mathrm I}_{s_n,s_m}'
  \end{align}
  where         
          $$
          {\mathrm I}_{s_n,s_m}' =
 {\mathbb{P}\left \{\underset{\theta\in \widetilde\Theta^{\nu}(s_n,s_m)}{\sup}\langle \xi_{\Omega},\theta-\theta^{*}\rangle-\frac{p}{8} \|\theta-\theta^{*} \|_2^2\geq 2\left (\sigma\vee\Q\right )^{2}R(s_n,s_m), \mathcal{E}_{1}\right \}}.
           $$
           Using a standard peeling argument  we find
            \begin{align}\label{stoch_2}
             \mathrm I_{s_n,s_m}' 
             &\leq  \sum_{l=7}^{\infty}\mathbb{P}\left \{\underset{\theta\in S_{l}}{\sup}\langle \xi_{\Omega},\theta-\theta^{*}\rangle-\frac{p}{8} \|\theta-\theta^{*} \|_2^2\geq 2\left (\sigma\vee\Q\right )^{2}R(s_n,s_m), \mathcal{E}_{1}\right \}\nonumber\\
             &\leq  \sum_{l=7}^{\infty}\mathbb{P}\left \{\underset{\theta\in \widetilde\Theta_{r}(s_n,s_m)}{\sup}\langle \xi_{\Omega},\theta-\theta^{*}\rangle\geq 2\,\left (\sigma\vee\Q\right )^{2}R(s_n,s_m)+\frac{pr^{2}}{8\mu^{2}}, \mathcal{E}_{1}\right \}
               \end{align}
             where $r= \mu^{l+1}\nu$. 
               Let $\{\widetilde G^{S}_{j}\}^{N_s}_{j=1}$ be a minimal $2^{-S}r$-covering set of $\Theta_{r}(s_n,s_m)$ in the Frobenius norm
              and $$S=\min\left \{s\geq 1\;:\;2^{-s}\leq \sqrt{\frac{\mathcal{I}}{nm}}\right \}.$$ As in the proof of Lemma \ref{rip_sampling} we define $G^{S}_{j}=\Pi(\widetilde G^{S}_{j})$. Then, there exists $G_{\theta}^{S}\in \{G^{S}_{j}\}^{N_s}_{j=1}$ such that on the event $\mathcal{E}_{1}$
             \begin{align*}
              \left \vert\sum_{(ij)}E_{ij}\xi_{ij}\left [\left (\theta-\theta^{*}\right )_{ij}-\left (G_{\theta}^{S}-\theta^{*}\right )_{ij}\right ]\right \vert\leq \frac{pr^{2}}{16\mu^{2}}
             \end{align*}
            where we use $\sqrt{\frac{\mathcal{I}}{nm}}\leq \frac{pr}{16\mu^{2}\sigma\sqrt{6pnm}}$. So, it suffices to prove the exponential inequality for
             \begin{align*}
             \mathbb{P}\left \{\underset{k=1,\dots,N_{S}}{\max}\sum_{(ij)}E_{ij}\xi_{ij}(G^{S}_{k}-\theta^{*})_{ij}\geq 2\, \left (\sigma\vee\Q\right )^{2}R(s_n,s_m)+\frac{pr^{2}}{16\mu^{2}}\right \}.
             \end{align*}
             We apply Markov's inequality. Set $t= R(s_n,s_m)+\frac{pr^{2}}{32\left (\sigma\vee\Q\right )^{2}\mu^{2}}$, then
             \begin{align}\label{stoch_3}
               &\mathbb{P}\left \{\underset{k=1,\dots,N_{S}}{\max}\dfrac{1}{2\left (\sigma\vee\Q\right )^{2}}\sum_{(ij)}E_{ij}\xi_{ij}(G^{S}_{k}-\theta^{*})_{ij}\geq t\right \}\nonumber\\&\hskip 0.5 cm\leq e^{-t}\E \exp\left (\underset{k=1,\dots,N_{S}}{\max}\dfrac{1}{2\left (\sigma\vee\Q\right )^{2}}\sum_{(ij)}E_{ij}\xi_{ij}(G^{S}_{k}-\theta^{*})_{ij}\right )\nonumber\\&\hskip 1 cm \leq e^{-t}\sum_{k=1}^{N_{S}}\E \exp\left (\dfrac{1}{2\left (\sigma\vee\Q\right )^{2}}\sum_{(ij)}E_{ij}\xi_{ij}(G^{S}_{k}-\theta^{*})_{ij}\right )
              \nonumber \\&\hskip 1.5 cm \leq e^{-t}\sum_{k=1}^{N_{S}}\prod_{(ij)}\E \exp\left (E_{ij}\xi_{ij}\dfrac{(G^{S}_{k}-\theta^{*})_{ij}}{2\left (\sigma\vee\Q\right )^{2}}\right ).
               \end{align}
               Lemma \ref{lm:exp_bernoulli} implies that $\mathbb{E}\left (e^{\lambda E_{ij}}\right )\leq e^{2\lambda p}$ for any $0\leq \lambda\leq 1$.
                Then, using
                 Assumption \ref{assumption_noise}  and \eqref{stoch_3} we get \begin{align*}
                   &\mathbb{P}\left \{\underset{k=1,\dots,N_{S}}{\max}\dfrac{1}{2\left (\sigma\vee\Q\right )^{2}}\sum_{(ij)}E_{ij}\xi_{ij}(G^{S}_{k}-\theta^{*})_{ij}\geq t\right \}
                   \\&\hskip 2 cm \leq e^{-t}\sum_{k=1}^{N_{S}}\prod_{(ij)} \exp\left (\dfrac{p(G^{S}_{k}-\theta^{*})^{2}_{ij}}{4\left (\sigma\vee\Q\right )^{2}}\right )\\&\hskip 2 cm\leq   \exp\left (\dfrac{pr^{2}}{2^{8}\left (\sigma\vee\Q\right )^{2}}+\log(N_{S})-t\right )\leq \exp\left (-\dfrac{3pr^{2}}{128\mu^{2}\left (\sigma\vee\Q\right )^{2}}\right ).
                   \end{align*}
               The last inequality follows from  Lemma \ref{lemma_control_covering} which implies $\log(N_{S})\leq R(s_n,s_m)$. Combining the last display with \eqref{stoch_1} and \eqref{stoch_2} 
%               and using that $W$ is a matrix with elements $W_{ij}= \theta_{ij}^*(E_{ij}-p)/p +\xi_{ij}E_{ij}/p$ 
               we find
                              \begin{align*}
                &\mathbb{P}\left \{\underset{\theta\in\mathcal{X}_{\nu}}{\sup}\dfrac{\langle \xi_{\Omega},\theta-\theta^{*}\rangle-\frac{p}{8} \|\theta-\theta^{*} \|_2^2}{\left (\sigma\vee\Q\right )^{2}R( \theta)}\geq 2, \mathcal{E}_{1}\right \}\\&\hskip 0.5 cm\leq \sum_{s_n=1}^{k_n}\sum_{s_m=1}^{k_m}\sum_{l=7}^{\infty} \exp\left (-\dfrac{3l\log(2)\mathcal{I}}{64}\right )\\&\hskip 1 cm\leq 2\sum_{s_n=1}^{k_n}\sum_{s_m=1}^{k_m}\exp\left (-\dfrac{\mathcal{I}}{5}\right )\leq 2\exp\left (-\dfrac{n+m}{10}\right )
               \end{align*}
               where we have used that $e^{x}\geq x$, $\nu^{2}=\frac{\left (\sigma\vee\Q\right )^{2}}{p}\mathcal{I}$, $\mathcal{I}=\underset{(s_n,s_m)}{\min} R(s_n,s_m)\geq (n+m)$, and $n\,m\log\left (3\sqrt{nm}\right )\geq 6\log \left (k_n\,k_m\right )$.
           \end{proof}   
           % % % % % % % % % % % % % % % % % % % % % % % % % % % % % % % % % % % % % % % % % % % % % % % % % % % % % % % % %
           \begin{lemma}\label{lemma_bernstein_1}
               Let $a=(a_{ij})\in\mathbb{R}^{n\times m}$. Then, for any $t>0$ we have
               \begin{align}\label{eq_bernstein_1}
               \mathbb{P} \left\{ \sum_{ij} a_{ij}E_{ij} \xi_{ij} > t \right\} \le \exp \left( -  \min\left\{\frac{t^2}{4\sigma^{2}p\|a\|_2^2}, \frac{t}{\sqrt{2}\sigma\|a\|_\infty} \right\} \right),
               \end{align}
               and 
               \begin{equation}\label{eq_bernstein_2}
               \mathbb{P} \left \{ \sqrt{\sum_{ij} E_{ij} \xi^{2}_{ij}} \ge \sigma\sqrt{6pnm} \right\} \le \exp ( -pnm).
               \end{equation}
               %where $N=pnm$.
               \end{lemma}
              \begin{proof}
Lemma \ref{lm:exp_bernoulli} implies that $\mathbb{E}\left (e^{\lambda E_{ij}}\right )\leq e^{2\lambda p}$ for any $0\leq \lambda\leq 1$. Then, using Assumption \ref{assumption_noise} for $0\leq\lambda\leq \sqrt{2}/(\sigma \Vert a\Vert_{\infty}) $ we obtain
  \begin{align*}
                \mathbb{E}\left [\exp(\lambda E_{ij}a_{ij}\xi_{ij})\right ]\leq \mathbb{E}\left [\exp(\lambda^{2} E_{ij}a^{2}_{ij}\sigma^{2}/2)\right ]\leq \exp(\lambda^{2} a^{2}_{ij}\sigma^{2}p).
                \end{align*}
            The Chernoff argument yields 
            \begin{align*}
                           \mathbb{P} \left\{ \sum_{ij} a_{ij}E_{ij} \xi_{ij} > t \right\}&=\mathbb{P} \left\{ e^{\sum_{ij} (\lambda a_{ij}E_{ij} \xi_{ij})} > e^{\lambda t}\right\}\\&\leq e^{-\lambda t} \exp(\lambda^{2} \sigma^{2}p\Vert a\Vert_{2}^{2}). 
                           \end{align*}
                           Now, choosing $\lambda=\frac{t}{2\sigma^{2}p\Vert a\Vert_{2}^{2}}$ if $t\leq \frac{2\sqrt{2}\sigma p \Vert a\Vert_{2}^{2}}{\Vert a\Vert_{\infty}}$ and $\lambda=\frac{\sqrt{2}}{\sigma\Vert a\Vert_{\infty}}$ if $t> \frac{2\sqrt{2}\sigma p \Vert a\Vert_{2}^{2}}{\Vert a\Vert_{\infty}}$ we get \eqref{eq_bernstein_1}.

                           We prove \eqref{eq_bernstein_2} in a similar way. Using Markov's inequality for $\lambda=\left (2\sigma^{2}\right )^{-1}$ we find
                           \begin{align*}
                                          \mathbb{P} \left \{ \sum_{ij} E_{ij} \xi^{2}_{ij} \ge t\right\}&= \mathbb{P} \left\{ e^{\lambda \sum_{ij} E_{ij} \xi^{2}_{ij}} > e^{\lambda t}\right\} \\&\leq e^{-\lambda t} \Pi_{ij}\mathbb{E}\exp(\lambda E_{ij} \left (\xi^{2}_{ij}- \sigma^{2}\right )+\lambda E_{ij}\sigma^{2})
                                          \\&\leq e^{-\lambda t} \Pi_{ij}\mathbb{E}\exp(2\lambda^{2} E_{ij}  \sigma^{4}+\lambda E_{ij}\sigma^{2})
                                          \\&\leq e^{- t/(2\sigma^{2})+2pnm}
                                          \end{align*}
               and we take $t=6\sigma^{2}pnm$.            
              \end{proof} 
              % % % % % % % % % % % % % % % %
              \begin{lemma}\label{lemma_control_covering}
              Let 
              $N_{S}=\mathcal{N}_{r2^{-S}}\left (\Theta_{r}(s_n,s_m)\right )$ where $$S=\min\left \{s\geq 3\;:\;2^{-s}\leq \sqrt{\frac{\mathcal{I}}{nm}}\right \}\quad \text{and }\quad r=2^{l+1}(\sigma\vee\Q) \sqrt{\frac{\mathcal{I}}{p}}$$ for $l\geq 6$. We have that $$R(s_n,s_m)\geq \log N_{S}.$$ 
              \end{lemma}
                
    \begin{proof}
    We use Proposition \ref{prop:coveringbound} and $\sqrt{\dfrac{\mathcal{I}}{nm}}\geq 2^{-S}\geq \dfrac{1}{2}\sqrt{\dfrac{\mathcal{I}}{nm}}$ to get
    \begin{eqnarray*}
    R_1(r2^{-S}) &\leq& ns_n \log \frac{ek_n}{s_n}+ m s_m \log \frac{ek_m}{s_m} + (ns_n+ms_m) \log\left ( \frac{3B_{max} s_m s_n}{2^{6}(\sigma\vee \Q)}\dfrac{nm\sqrt{p}}{\mathcal{I}}\right )\\ &+ & r_n r_m \log\left ( 9\sqrt{\dfrac{nm}{\mathcal{I}}}\right )\\
    R_2(r2^{-S}) &\leq& nr_m \log\left ( 6\sqrt{\dfrac{nm}{\mathcal{I}}}\right ) + ms_m \log \frac{ek_m}{s_m} + ms_m \log\left ( \frac{B_{max} s_m s_n}{2^{6}(\sigma\vee \Q)}\dfrac{nm\sqrt{p}}{\mathcal{I}}\right ),\\
    R_3(r2^{-S}) &\leq& mr_n \log\left ( 6\sqrt{\dfrac{nm}{\mathcal{I}}}\right )+ ns_n \log \frac{ek_n}{s_n} + ns_n  \log\left ( \frac{B_{max} s_m s_n}{2^{6}(\sigma\vee \Q)}\dfrac{nm\sqrt{p}}{\mathcal{I}}\right ),\\
    R_4(r2^{-S}) &\leq& m n \log\left ( 3\sqrt{\dfrac{nm}{\mathcal{I}}}\right )
    \end{eqnarray*}
    and we have that 
    \begin{equation*}
    R_1(r2^{-S})\wedge R_2(r2^{-S}) \wedge R_3(r2^{-S}) \wedge R_4(r2^{-S})\leq R(s_n,s_m)
    \end{equation*}
    where we have used that $\mathcal{I}\geq n+m$.
     \end{proof}       
     
% % % % % % % % % % % % % % % % % % % % % % % % %
%\input{prooflowerbound.tex}

% % % % % % % % % % % % % % % % % % % % % % % % % % % % % % % % % % % % % % % % % % % % % 
% % % % % % % % % % % % % % % % % % % % % % % % % % % % % % % % % % %

 % % % % % % % % % % % % % % % % % % % % % % % % % % % % % % % % % % % % % % % % % % % % % % %

%\bibliographystyle{plain}

%\bibliography{biblio_olga}
\end{document}